\documentclass{imanum}
\usepackage{graphicx}
\usepackage{enumerate}
\jno{drnxxx}
\allowdisplaybreaks[4]
\begin{document}

\title{Regularity theory and numerical algorithm for the fractional  Klein-Kramers equation}
\shorttitle{NUMERICAL ALGORITHM FOR THE FRACTIONAL  KLEIN-KRAMERS EQUATION}

\author{%
{\sc
Jing Sun\thanks{Email: js@lzu.edu.cn}
and
Daxin Nie\thanks{Email: ndx1993@163.com}
and
Weihua Deng\thanks{Corresponding author. Email: dengwh@lzu.edu.cn}
} \\[2pt]
School of Mathematics and Statistics, Gansu Key Laboratory of Applied Mathematics and\\ Complex Systems, Lanzhou University, Lanzhou 730000, P.R. China}

\shortauthorlist{J. Sun, D.X. Nie and W.H. Deng
}

\maketitle

\begin{abstract}
{Fractional Klein-Kramers equation can well describe subdiffusion in phase space. In this paper, we develop the fully discrete scheme for fractional Klein-Kramers equation based on the backward Euler convolution quadrature and local discontinuous Galerkin methods. Thanks to the obtained sharp regularity estimates in temporal and spatial directions after overcoming the hypocoercivity of the operator, the complete error analyses of the fully discrete scheme are built. 
It's worth mentioning that the convergence of the provided scheme is independent of the temporal regularity of the exact solution. Finally, numerical results are proposed to verify the correctness of the theoretical results.}
{fractional Klein-Kramers equation; regularity estimate;  convolution quadrature; local discontinuous Galerkin method; error analysis}
\end{abstract}

\section{Introduction}
Subdiffusion is ubiquitous in the nature world \citep[see, e.g.,][]{Metzler.2000Trwsgtadafda}. 
Microscopically, it can be modeled by Langevin dynamics with long-tailed trapping \citep[see, e.g.,][]{Metzler.2000Stctteftletfd}. To describe how the presence
of the trapping events leads to the macroscopic observation of subdiffusion, the authors establish the fractional Klein-Kramers equation \citep[see, e.g.,][]{Metzler.2000FaGCKEttFKKE,Metzler.2000Stctteftletfd}. This paper is concerned with the regularity estimate and numerical analysis for the fractional Klein-Kramers equation, i.e.,
\begin{equation}\label{eqretosol}
	\begin{aligned}
		&\partial_{t}G(x,v,t)+{}_{0}\partial_{t}^{1-\alpha}\left (\gamma v\frac{\partial}{\partial x}-\gamma\frac{\partial }{\partial v}\eta v-\frac{\gamma \eta}{m\beta}\frac{\partial^{2}}{\partial v^{2}}\right )G(x,v,t)\\
		&\qquad \qquad\qquad \qquad\qquad \qquad\qquad \qquad\qquad \qquad ={}_{0}\partial_{t}^{1-\alpha}f \qquad ((x,v),t)\in\Omega\times (0,T]
	\end{aligned}
\end{equation}
with initial condition
\begin{equation*}
	G(x,v,0)=G_{0}\qquad (x,v)\in\Omega
\end{equation*}
and boundary conditions
\begin{equation}\label{eqbdcond}
	\begin{aligned}
		&G(x,0,t)=G(x,1,t)=0 & \qquad (x,t)\in(0,1)\times (0,T],\\
		&G(0,v,t)=0 & \qquad (v,t)\in(0,1)\times(0,T].
	\end{aligned}
\end{equation}
Here $\Omega=\{(x,v)~|~0<x<1,~0<v<1\}$; $T$ denotes the fixed terminal time; $f$ is source term; $v$ is the velocity, $\eta$ is the friction constant, $m$ is the mass of the particle, $\gamma$ is the ratio of the intertrapping time scale and the internal waiting time scale, and $\beta$ is a variable related to the temperature and Boltzmann's constant; without loss of generality, we take $\eta=\beta=m=\gamma=1$ in the following;  $~_0\partial^{1-\alpha}_t$ is the Riemann-Liouville fractional derivative with $\alpha\in(0,1)$ defined by \citep[see, e.g.,][]{Podlubny.1999Fde}
\begin{equation}
	_{0}\partial^{1-\alpha}_tG=\frac{1}{\Gamma(\alpha)}\frac{\partial}{\partial t}\int^t_{0}(t-\xi)^{\alpha-1}G(\xi)d\xi.
\end{equation}

In the past few years, there have been some discussions for solving fractional Klein-Kramers equation numerically \citep[see, e.g.,][]{Deng.2011Fdmatpcftfkke,Gao.2012AfdaftibvpotfKKeips,Li.2012FdaadsftLFKKe,Nikan.2021NeotfKKmaimd}. In \citet{Deng.2011Fdmatpcftfkke}, the authors consider the finite difference scheme for the fractional Klein-Kramers equation and provide the corresponding error analyses; furthermore the authors use finite difference scheme to solve fractional Klein-Kramers equation with Riesz fractional derivative in \citet{Li.2012FdaadsftLFKKe} and  \citet{Nikan.2021NeotfKKmaimd} provide a hybrid algorithm using the local radial basis functions based on  finite difference to obtain the numerical solution of the fractional Klein-Kramers equation. From the above works, it can be noted that the corresponding numerical discussions in Galerkin framework for fractional  Klein-Kramers equation are rare.

In this paper, we first build the regularity of Eq. \eqref{eqretosol}, and then present the robust numerical scheme and complete error analyses. As for the regularity estimates, to overcome the challenges caused by the hypocoercivity of the operator $\left (v\frac{\partial}{\partial x}-\frac{\partial }{\partial v} v-\frac{\partial^{2}}{\partial v^{2}}\right )$, we introduce a new operator $\mathcal{L}$ (one can refer to \eqref{eqdefL}) and provide the corresponding resolvent estimate (see Lemma \ref{lemresest}); with the help of equivalent form of \eqref{eqretosol} and resolvent estimate, we find
\begin{equation*}
	\|G(t)\|_{L^{2}(\Omega)}\leq C\|G_{0}\|_{L^{2}(\Omega)}+C\|f(0)\|_{L^{2}(\Omega)}+C\int_{0}^{t}\|f'(s)\|_{L^{2}(\Omega)}ds
\end{equation*}
and
\begin{equation*}
	\|G'(t)\|_{L^{2}(\Omega)}\leq Ct^{-1}\|G_{0}\|_{L^{2}(\Omega)}+Ct^{\alpha-1}\|f(0)\|_{L^{2}(\Omega)}+C\int_{0}^{t}(t-s)^{\alpha-1}\|f'(s)\|_{L^{2}(\Omega)}ds.
\end{equation*}
Next we use backward Euler convolution quadrature to discretize the temporal derivative and an $\mathcal{O}(\tau)$ convergence rate is obtained without any regularity assumptions on the exact solution. At last, we use local discontinuous Galerkin method to discretize spatial derivative; to obtain the stability and the convergence of the fully-discrete scheme, we build a new discrete Gr\"onwall's inequality (see Lemma \ref{lemstab} for the details).

The plan of the paper is as follows. Next, we provide the regularity of the fractional Klein-Kramers equation in temporal  and spatial directions, respectively. In Section 3, the time semi-discrete scheme is built by backward Euler convolution quadrature and the resulting error analysis is also provided. Then we use the local discontinuous Galerkin method to  discretize the space operator and the error estimate is obtained in Section 4. Section 5 validates the effectiveness of the algorithm by extensive numerical experiments. We conclude the paper with some discussions in the last section. Throughout the paper, $C$ is  a positive constant, whose value may differ at different places; $\|\cdot\|$ stands for the operator norm from $L^{2}(\Omega)$ to $L^{2}(\Omega)$; and `$\tilde{~}$' means Laplace transform.

\section{Regularity of the solution}
Here we first provide some notations, and then present the solution and discuss its regularity.
Define $\Gamma_{\theta,\kappa}$ for $\kappa>0$ and $\theta\in(\frac{\pi}{2},\pi)$ as
\begin{equation*}
	\Gamma_{\theta,\kappa}=\{r e^{-\mathbf{i}\theta}: r\geq \kappa\}\cup\{\kappa e^{\mathbf{i}\omega}: |\omega|\leq \theta\}\cup\{r e^{\mathbf{i}\theta}: r\geq \kappa\},
\end{equation*}
where the circular arc is oriented counterclockwise and the two rays are oriented with an increasing imaginary part. Here $\mathbf{i}$ denotes the imaginary unit.
Define the operator $\mathcal{L}$ as
\begin{equation}\label{eqdefL}
	\mathcal{L}= v\frac{\partial}{\partial x}-\frac{\partial }{\partial v} v-\frac{\partial^{2}}{\partial v^{2}}+\frac{1}{2}
\end{equation}
with boundary conditions \eqref{eqbdcond}.

Applying ${}_{0}\partial^{\alpha-1}_{t}$ on both sides of \eqref{eqretosol} and using the definition of $\mathcal{L}$, we can get the equivalent form of \eqref{eqretosol}, i.e., 
\begin{equation}\label{eqretosol1}
	\left\{\begin{aligned}
		&{}_{0}\partial_{t}^{\alpha}(G(x,v,t)-G_{0})+\mathcal{L}G(x,v,t)=f+\frac{1}{2}G(x,v,t) & \qquad ((x,v),t)\in\Omega\times (0,T],\\
		&G(x,v,0)=G_{0} & \qquad (x,v)\in\Omega,\\
		&G(x,0,t)=G(x,1,t)=0 & \qquad (x,t)\in(0,1)\times (0,T],\\
		&G(0,v,t)=0  & \qquad (v,t)\in(0,1)\times(0,T].
	\end{aligned}\right.
\end{equation}
As for the operator $\mathcal{L}$, according to its definition \eqref{eqdefL} and using integration by parts, it's easy to check that
\begin{equation*}
	\begin{aligned}
		&(\mathcal{L} G,G)\\
		=&\int_{0}^{1}v\int_{0}^{1}\frac{\partial }{\partial x}G(x,v,t)G(x,v,t)dxdv-\int_{0}^{1}\int_{0}^{1}\left (\frac{\partial }{\partial v}(vG(x,v,t))\right )G(x,v,t)dvdx\\
		&-\int_{0}^{1}\int_{0}^{1}\frac{\partial^{2}}{\partial v^{2}}G(x,v,t)G(x,v,t)dvdx+\frac{1}{2}\int_{0}^{1}\int_{0}^{1}G(x,v,t)G(x,v,t)dxdv\\
		=&\frac{1}{2}\int_{0}^{1}v\int_{0}^{1}\frac{\partial }{\partial x}(G(x,v,t))^{2}dxdv-\int_{0}^{1}\int_{0}^{1}\left (v\frac{\partial }{\partial v}G(x,v,t)+G(x,v,t)\right )G(x,v,t)dvdx\\
		&+\int_{0}^{1}\int_{0}^{1}\frac{\partial}{\partial v}G(x,v,t)\frac{\partial}{\partial v}G(x,v,t)dvdx+\frac{1}{2}\int_{0}^{1}\int_{0}^{1}G(x,v,t)G(x,v,t)dxdv\\
		=&\frac{1}{2}\int_{0}^{1}v(G(1,v,t))^{2}dv-\frac{1}{2}\int_{0}^{1}\int_{0}^{1}v\frac{\partial }{\partial v}(G(x,v,t))^{2}dvdx\\
		&+\int_{0}^{1}\int_{0}^{1}\frac{\partial}{\partial v}G(x,v,t)\frac{\partial}{\partial v}G(x,v,t)dvdx-\frac{1}{2}\int_{0}^{1}\int_{0}^{1}G(x,v,t)^{2}dvdx,
	\end{aligned}
\end{equation*}
where $(\cdot,\cdot)$ means the inner product on $\Omega$.
Since $G(x,0,t)=G(x,1,t)=0$, it follows that
\begin{equation*}
	\begin{aligned}
		&\int_{0}^{1}\int_{0}^{1}v\frac{\partial }{\partial v}G(x,v,t)^{2}dvdx\\
		=&\int_{0}^{1}\int_{0}^{1}vdG(x,v,t)^{2}dx\\
		=&\int_{0}^{1}vG(x,v,t)^{2}|_{v=0}^{v=1}dx-\int_{0}^{1}\int_{0}^{1}G(x,v,t)^{2}dvdx\\
		=&-\int_{0}^{1}\int_{0}^{1}G(x,v,t)^{2}dvdx,
	\end{aligned}
\end{equation*}
implying
\begin{equation}\label{eqbinformL}
	\begin{aligned}
		(\mathcal{L} G,G)
		=&\frac{1}{2}\int_{0}^{1}vG(1,v,t)^{2}dv+\int_{0}^{1}\int_{0}^{1}\frac{\partial}{\partial v}G(x,v,t)\frac{\partial}{\partial v}G(x,v,t)dvdx\geq 0,
	\end{aligned}
\end{equation}
namely, $\mathcal{L}$ is a positive operator.

Next, to get the regularity of the solution of \eqref{eqretosol1} and using $f(t)=f(0)+\int_{0}^{t}f'(s)ds$, we present the solution in Laplace space, i.e.,
\begin{equation}\label{eqsollap}
	\tilde{G}(z)=\tilde{E}(z)G_{0}+\tilde{\mathcal{R}}(z)f(0)+\tilde{\mathcal{R}}(z)\tilde{f'}(z)+\frac{1}{2}\tilde{F}(z)\tilde{G}(z),
\end{equation}
where `$\tilde{~}$' means Laplace transform and
\begin{equation*}
	\begin{aligned}
		&\tilde{E}(z)=z^{\alpha-1}(z^{\alpha}+\mathcal{L})^{-1},\\
		&\tilde{F}(z)=(z^{\alpha}+\mathcal{L})^{-1},\\
		&\tilde{\mathcal{R}}(z)=z^{-1}(z^{\alpha}+\mathcal{L})^{-1}.
	\end{aligned}
\end{equation*}
Thus, iterating \eqref{eqsollap} $m$ $(m\in\mathbb{N})$ times, one can get
\begin{equation}\label{eqsollap2}
	\tilde{G}(z)=\sum_{k=0}^{m}\left (\frac{1}{2}\tilde{F}(z)\right )^{k}\left(\tilde{E}(z)G_{0}+\tilde{\mathcal{R}}(z)f(0)+\tilde{\mathcal{R}}(z)\tilde{f'}(z)\right)+\left(\frac{1}{2}\tilde{F}(z)\right)^{m+1}\tilde{G}(z).
\end{equation}

Now we present the resolvent estimate of $\mathcal{L}$.

\begin{lemma}\label{lemresest}
	Let $\mathcal{L}$ be defined in \eqref{eqdefL}. Then for $z\in \Sigma_{\theta}=\{z\in\mathbb{C}:z\neq 0,|\arg z|\leq \theta\}$ with $\theta\in (\frac{\pi}{2},\pi)$, there holds
	\begin{equation*}
		\|(z+\mathcal{L})^{-1}\|\leq C|z|^{-1}.
	\end{equation*}
\end{lemma}
\begin{proof}
	Let $(z+\mathcal{L})u=g$. Then we have
	\begin{equation*}
		\begin{aligned}
			((z+\mathcal{L})u,u)=(g,u).
		\end{aligned}
	\end{equation*}
	According to \eqref{eqbinformL}, one obtains
	\begin{equation*}
		\begin{aligned}
			((z+\mathcal{L})u,u)\geq&|z|\|u\|_{L^{2}(\Omega)}^{2}.
		\end{aligned}
	\end{equation*}
	Using
	\begin{equation*}
		(g,u)\leq C\|g\|_{L^{2}(\Omega)}\|u\|_{L^{2}(\Omega)}
	\end{equation*}
leads to the desired result.
\end{proof}
\begin{remark}
	By Lemma \ref{lemresest}, the inverse Laplace transforms of $\tilde{E}(z)$, $\tilde{F}(z)$, and $\tilde{F}(z)$ satisfy
	\begin{equation}\label{eqoptEest}
		\begin{aligned}
			\|E(t)\|\leq & C,\quad
			\|F(t)\|\leq Ct^{\alpha-1},\quad \|\mathcal{R}(t)\|\leq Ct^{\alpha}.
		\end{aligned}
	\end{equation}
\end{remark}

Then we discuss spatial regularity of the solution $G$.
\begin{theorem}\label{thmspareg}
	Let $G$ be the solution of \eqref{eqretosol1}. Assume $G_{0},~f(0)\in L^{2}(\Omega)$, and $\int_{0}^{t}\|f'(s)\|_{L^{2}(\Omega)}ds<\infty$. Then one has
	\begin{equation*}
		\|G(t)\|_{L^{2}(\Omega)}\leq C\|G_{0}\|_{L^{2}(\Omega)}+C\|f(0)\|_{L^{2}(\Omega)}+C\int_{0}^{t}\|f'(s)\|_{L^{2}(\Omega)}ds.
	\end{equation*}
\end{theorem}
\begin{proof}
	First, taking inverse Laplace transform for \eqref{eqsollap} leads to
	\begin{equation}\label{eqsolrep}
		G(t)=E(t)G_{0}+\mathcal{R}(t)f(0)+\int_{0}^{t}\mathcal{R}(t-s)f'(s)ds+\frac{1}{2}\int_{0}^{t}F(t-s)G(s)ds.
	\end{equation}
	According to \eqref{eqoptEest} and using $T/t\geq 1$, we have
	\begin{equation*}
		\begin{aligned}
			&\|G(t)\|_{L^{2}(\Omega)}\\\leq&\|E(t)G_{0}\|_{L^{2}(\Omega)}+\left\|\mathcal{R}(t)f(0)\right \|_{L^{2}(\Omega)}+\left \|\int_{0}^{t}\mathcal{R}(t-s)f'(s)ds\right \|_{L^{2}(\Omega)}+\left \|\frac{1}{2}\int_{0}^{t}F(t-s)G(s)ds\right \|_{L^{2}(\Omega)}\\
			\leq& C\|G_{0}\|_{L^{2}(\Omega)}+C\|f(0)\|_{L^{2}(\Omega)}+C\int_{0}^{t}\|f'(s)\|_{L^{2}(\Omega)}ds+C\int_{0}^{t}(t-s)^{\alpha-1}\|G(s)\|_{L^{2}(\Omega)}ds;
		\end{aligned}
	\end{equation*}
By Gr\"{o}nwall's inequality \citep[see, e.g.,][]{Elliott.1992EewsandfafemftCHe}, there exists 
	\begin{equation*}
		\|G(t)\|_{L^{2}(\Omega)}\leq C\|G_{0}\|_{L^{2}(\Omega)}+C\|f(0)\|_{L^{2}(\Omega)}+C\int_{0}^{t}\|f'(s)\|_{L^{2}(\Omega)}ds.
	\end{equation*}
	The proof is completed.
\end{proof}

Next, we provide the estimate of $\|G'(t)\|_{L^2(\Omega)}$, where $G'(t)$ denotes the first derivative of $G$ with respect to $t$.
\begin{theorem}\label{thmholdere}
	Let $G$ be the solution of \eqref{eqretosol1}. Assume $G_{0},~f(0)\in L^{2}(\Omega)$, and $\int_{0}^{t}\|f'(s)\|_{L^{2}(\Omega)}ds<\infty$. Then one has
	\begin{equation*}
		\|G'(t)\|_{L^{2}(\Omega)}\leq Ct^{-1}\|G_{0}\|_{L^{2}(\Omega)}+Ct^{\alpha-1}\|f(0)\|_{L^{2}(\Omega)}+C\int_{0}^{t}(t-s)^{\alpha-1}\|f'(s)\|_{L^{2}(\Omega)}ds.
	\end{equation*}
\end{theorem}
\begin{proof}
	Let $\left \lfloor\frac{1}{\alpha}\right \rfloor$ be the largest integer smaller than $\frac{1}{\alpha}$. Choosing $m=\left \lfloor\frac{1}{\alpha}\right \rfloor$ and using \eqref{eqsollap2}, we have
	\begin{equation*}
		\begin{aligned}
			&\left\|\lim_{\tau\rightarrow 0}\frac{G(t)-G(t-\tau)}{\tau}\right\|_{L^{2}(\Omega)}\\
			\leq&\sum_{k=0}^{m}C\left\|\lim_{\tau\rightarrow 0}\frac{\int_{\Gamma_{\theta,\kappa}}e^{zt}(1-e^{-z\tau})\left(\frac{1}{2}\tilde{F}\right)^{k}\tilde{E}dz}{\tau}G_{0}\right \|_{L^{2}(\Omega)}\\
			&+\sum_{k=0}^{m}C\left\|\lim_{\tau\rightarrow 0}\frac{\int_{\Gamma_{\theta,\kappa}}e^{zt}(1-e^{-z\tau})\left(\frac{1}{2}\tilde{F}\right)^{k}\tilde{\mathcal{R}}dz}{\tau}f(0)\right \|_{L^{2}(\Omega)}\\
			&+\sum_{k=0}^{m}C\left\|\lim_{\tau\rightarrow 0}\frac{\int_{\Gamma_{\theta,\kappa}}e^{zt}(1-e^{-z\tau})\left(\frac{1}{2}\tilde{F}\right)^{k}\tilde{\mathcal{R}}\tilde{f'}dz}{\tau}\right \|_{L^{2}(\Omega)}\\
			&+C\left\|\lim_{\tau\rightarrow 0}\frac{\int_{\Gamma_{\theta,\kappa}}e^{zt}(1-e^{-z\tau})\left(\frac{1}{2}\tilde{F}\right)^{m+1}\tilde{G}dz}{\tau}\right \|_{L^{2}(\Omega)}\\
			\leq&\sum_{k=0}^{m}\uppercase\expandafter{\romannumeral1}_{k}+\sum_{k=0}^{m}\uppercase\expandafter{\romannumeral2}_{k}+\sum_{k=0}^{m}\uppercase\expandafter{\romannumeral3}_{k}+\uppercase\expandafter{\romannumeral4}.
		\end{aligned}
	\end{equation*}
	As for $\uppercase\expandafter{\romannumeral1}_{k}$, Lemma \ref{lemresest}, the definitions of operators $E$ and $F$, and the fact $\left|\frac{1-e^{-z\tau}}{\tau}\right|\leq C|z|$ for $z\in\Gamma_{\theta,\kappa}$ \citep[see, e.g.,][]{Lubich.1996Ndeefaoaeewaptmt} give
	\begin{equation*}
		\begin{aligned}
			\uppercase\expandafter{\romannumeral1}_{k}\leq&C\int_{\Gamma_{\theta,\kappa}}|e^{zt}||z||z|^{-\alpha k}|z|^{-1}|dz|\|G_{0}\|_{L^2(\Omega)}\\
			\leq&Ct^{\alpha k-1}\|G_{0}\|_{L^2(\Omega)}\\
			\leq& Ct^{-1}\|G_{0}\|_{L^2(\Omega)}.
		\end{aligned}
	\end{equation*}
	Similarly, one can get
	\begin{equation*}
		\begin{aligned}
			&\uppercase\expandafter{\romannumeral2}_{k}\leq Ct^{\alpha-1}\|f(0)\|_{L^2(\Omega)},\\
			&\uppercase\expandafter{\romannumeral3}_{k}\leq C\int_{0}^{t}(t-s)^{\alpha-1}\|f'(s)\|_{L^2(\Omega)}ds.	
		\end{aligned}
	\end{equation*}
	As for $\uppercase\expandafter{\romannumeral4}$, using $m=\left \lfloor\frac{1}{\alpha}\right \rfloor$ and Theorem \ref{thmspareg}, we obtain
	\begin{equation*}
		\begin{aligned}
			\uppercase\expandafter{\romannumeral4}\leq&C\int_{0}^{t}\int_{\Gamma_{\theta,\kappa}}|e^{z(t-s)}||z||z|^{-(m+1)\alpha}|dz|\left\|G(s)\right \|_{L^{2}(\Omega)}ds\\
			\leq& C\int_{0}^{t}(t-s)^{(m+1)\alpha-2}\left\|G(s)\right \|_{L^{2}(\Omega)}ds\\
			\leq&C\|G_{0}\|_{L^{2}(\Omega)}+C\|f(0)\|_{L^{2}(\Omega)}+C\int_{0}^{t}\|f'(s)\|_{L^{2}(\Omega)}ds.
		\end{aligned}
	\end{equation*}
	Collecting the above estimates and using $T/t\geq 1$ lead to the desired result.  
\end{proof}

\section{Temporal semi-discrete scheme and error analysis}
In this section, we consider the temporal semi-discrete scheme of \eqref{eqretosol1} and provide the corresponding error estimate. Let $\tau=T/L$  ($L\in\mathbb{N}^{*}$) and $t_{i}=i\tau$, $i=1,2,\ldots,L$. Here we use backward Euler convolution quadrature \citep[see, e.g.,][]{Lubich.1988CqadocI,Lubich.1988CqadocIb,Lubich.1996Ndeefaoaeewaptmt} to discretize temporal operator, i.e., the temporal semi-discrete scheme can be written as: Find $G^{n}$ such that
\begin{equation}\label{eqsemidissch}
	\left\{
	\begin{aligned}
		&\sum_{j=0}^{n-1}d^{(\alpha)}_{j}(G^{n-j}-G_{0})+\mathcal{L}G^{n}=f^{n}+\frac{1}{2}G^{n},\\
		&G^{n}(x,0)=G^{n}(x,1)=G^{n}(0,v)=0,\\
		&G^{0}=G_{0},
	\end{aligned}
	\right.
\end{equation}
where  $G^{n}$ is the numerical solution of $G(x,v,t_{n})$, $f^{n}=f(t_{n})$, and
\begin{equation}\label{eqdefdja}
	\sum_{j=0}^{\infty}d^{(\alpha)}_{j}\zeta^{j}=(\delta_{\tau}(\zeta))^{\alpha}=\left(\frac{1-\zeta}{\tau}\right )^{\alpha}.
\end{equation}
It's easy to know that $d^{\alpha}_{j}$ have the following properties.
\begin{lemma}\citep[see, e.g.,][]{Chen.2009FdaftfFPe,Deng.2015NAftFaBFFKE}\label{lempropb}
	Let $\{d^{(\alpha)}_{i}\}_{i=0}^{\infty}$ be defined in \eqref{eqdefdja} with $\alpha\in(0,1)$. Then we have
	\begin{equation*}
		\begin{aligned}
			&d^{(\alpha)}_{0}>0;\qquad d^{(\alpha)}_{i}<0,\quad i\geq 1;\\
			&\sum_{i=0}^{\infty}d^{(\alpha)}_{i}=0.
		\end{aligned}
	\end{equation*}
\end{lemma}

To obtain the temporal error estimate, we first derive the solution of semi-discrete scheme \eqref{eqsemidissch}. Multiplying $\zeta^{n}$ on both sides of \eqref{eqsemidissch} and summing $n$ from $1$ to $\infty$, we have
\begin{equation*}
	\sum_{n=1}^{\infty}\sum_{j=0}^{n-1}d^{(\alpha)}_{j}(G^{n-j}-G_{0})\zeta^{n}+\sum_{n=1}^{\infty}\mathcal{L}G^{n}\zeta^{n}=\sum_{n=1}^{\infty}f^{n}\zeta^{n}+\sum_{n=1}^{\infty}\frac{1}{2}G^{n}\zeta^{n}.
\end{equation*}
Using the fact $f(t)=f(0)+\int_{0}^{t}f'(s)ds$ and introducing $R(t)=\int_{0}^{t}f'(s)ds$ lead to
\begin{equation}\label{eqtimsollap}
	\begin{aligned}
		\sum_{n=1}^{\infty}G^{n}\zeta^{n}=&((	\delta_{\tau}(\zeta))^{\alpha}+\mathcal{L})^{-1}\sum_{n=1}^{\infty}f^{0}\zeta^{n}+((	\delta_{\tau}(\zeta))^{\alpha}+\mathcal{L})^{-1}\sum_{n=1}^{\infty}R(t_{n})\zeta^{n}\\
		&+((	\delta_{\tau}(\zeta))^{\alpha}+\mathcal{L})^{-1}(	\delta_{\tau}(\zeta))^{\alpha}\sum_{n=1}^{\infty}G^{0}\zeta^{n}+\frac{1}{2}((\delta_{\tau}(\zeta))^{\alpha}+\mathcal{L})^{-1}\sum_{n=1}^{\infty}G^{n}\zeta^{n}.
	\end{aligned}
\end{equation}
Iterating \eqref{eqtimsollap} $m$ $(m\in\mathbb{N})$ times, one has
\begin{equation}\label{eqtimsollap1}
	\begin{aligned}
		\sum_{n=1}^{\infty}G^{n}\zeta^{n}=&\sum_{k=0}^{m}\left(\frac{1}{2}((	\delta_{\tau}(\zeta))^{\alpha}+\mathcal{L})^{-1}\right )^{k}((	\delta_{\tau}(\zeta))^{\alpha}+\mathcal{L})^{-1}\sum_{n=1}^{\infty}f^{0}\zeta^{n}\\
		&+\sum_{k=0}^{m}\left(\frac{1}{2}((	\delta_{\tau}(\zeta))^{\alpha}+\mathcal{L})^{-1}\right )^{k}((	\delta_{\tau}(\zeta))^{\alpha}+\mathcal{L})^{-1}\sum_{n=1}^{\infty}R(t_{n})\zeta^{n}\\
		&+\sum_{k=0}^{m}\left(\frac{1}{2}((	\delta_{\tau}(\zeta))^{\alpha}+\mathcal{L})^{-1}\right )^{k}((	\delta_{\tau}(\zeta))^{\alpha}+\mathcal{L})^{-1}(	\delta_{\tau}(\zeta))^{\alpha}\sum_{n=1}^{\infty}G^{0}\zeta^{n}\\
		&+\left(\frac{1}{2}((	\delta_{\tau}(\zeta))^{\alpha}+\mathcal{L})^{-1}\right )^{m+1}\sum_{n=1}^{\infty}G^{n}\zeta^{n}.
	\end{aligned}
\end{equation}
Thus there holds
\begin{equation*}
	\begin{aligned}
		G^{n}=&\sum_{k=0}^{m}\frac{1}{2\pi\mathbf{i}\tau}\int_{|\zeta|=\xi_{\tau}}\zeta^{-n-1}\left(\frac{1}{2}((	\delta_{\tau}(\zeta))^{\alpha}+\mathcal{L})^{-1}\right )^{k}((	\delta_{\tau}(\zeta))^{\alpha}+\mathcal{L})^{-1}(\delta_{\tau}(\zeta))^{-1}\zeta d\zeta f(0)\\
		&+\sum_{k=0}^{m}\frac{1}{2\pi\mathbf{i}}\int_{|\zeta|=\xi_{\tau}}\zeta^{-n-1}\left(\frac{1}{2}((	\delta_{\tau}(\zeta))^{\alpha}+\mathcal{L})^{-1}\right )^{k}((	\delta_{\tau}(\zeta))^{\alpha}+\mathcal{L})^{-1}\sum_{j=1}^{\infty}R(t_{j})\zeta^{j}d\zeta \\
		&+\sum_{k=0}^{m}\frac{1}{2\pi\mathbf{i}\tau}\int_{|\zeta|=\xi_{\tau}}\zeta^{-n-1}\left(\frac{1}{2}((	\delta_{\tau}(\zeta))^{\alpha}+\mathcal{L})^{-1}\right )^{k}((	\delta_{\tau}(\zeta))^{\alpha}+\mathcal{L})^{-1}(	\delta_{\tau}(\zeta))^{\alpha-1}G^{0}\zeta d\zeta\\
		&+\frac{1}{2\pi\mathbf{i}}\int_{|\zeta|=\xi_{\tau}}\zeta^{-n-1}\left(\frac{1}{2}((	\delta_{\tau}(\zeta))^{\alpha}+\mathcal{L})^{-1}\right )^{m+1}\sum_{j=1}^{\infty}G^{j}\zeta^{j}d\zeta.
	\end{aligned}
\end{equation*}
Taking $\xi_{\tau}=e^{-(\kappa+1)\tau}$, using Cauchy's integral theorem, and introducing $\Gamma^{\tau}_{\theta,\kappa}=\{z\in\mathbb{C}:~\kappa\leq |z|\leq \frac{\pi}{\tau\sin(\theta)},~|\arg z|=\theta\}\cup\{z\in\mathbb{ C}:~|z|=\kappa,~|\arg z|\leq \theta\}$, one can get
\begin{equation}\label{eqtimschsol}
	\begin{aligned}
		G^{n}=&\sum_{k=0}^{m}\frac{1}{2\pi\mathbf{i}}\int_{\Gamma^{\tau}_{\theta,\kappa}}e^{zt_{n-1}}\left(\frac{1}{2}((	\delta_{\tau}(e^{-z\tau}))^{\alpha}+\mathcal{L})^{-1}\right )^{k}((	\delta_{\tau}(e^{-z\tau}))^{\alpha}+\mathcal{L})^{-1}(\delta_{\tau}(e^{-z\tau}))^{-1}dz f(0)\\
		&+\sum_{k=0}^{m}\frac{1}{2\pi\mathbf{i}}\int_{\Gamma^{\tau}_{\theta,\kappa}}e^{zt_{n}}\left(\frac{1}{2}((	\delta_{\tau}(e^{-z\tau}))^{\alpha}+\mathcal{L})^{-1}\right )^{k}((	\delta_{\tau}(e^{-z\tau}))^{\alpha}+\mathcal{L})^{-1}\tau\sum_{j=1}^{\infty}R(t_{j})e^{-zt_{j}}dz \\
		&+\sum_{k=0}^{m}\frac{1}{2\pi\mathbf{i}}\int_{\Gamma^{\tau}_{\theta,\kappa}}e^{zt_{n-1}}\left(\frac{1}{2}((	\delta_{\tau}(e^{-z\tau}))^{\alpha}+\mathcal{L})^{-1}\right )^{k}((	\delta_{\tau}(e^{-z\tau}))^{\alpha}+\mathcal{L})^{-1}(	\delta_{\tau}(e^{-z\tau}))^{\alpha-1}G^{0}dz\\
		&+\frac{1}{2\pi\mathbf{i}}\int_{\Gamma^{\tau}_{\theta,\kappa}}e^{zt_{n}}\left(\frac{1}{2}((	\delta_{\tau}(e^{-z\tau}))^{\alpha}+\mathcal{L})^{-1}\right )^{m+1}\tau\sum_{j=1}^{\infty}G^{j}e^{-zt_{j}}dz.
	\end{aligned}
\end{equation}

\begin{theorem}\label{thmsemierr}
	Let $G$ and $G^{n}$ be the solutions of \eqref{eqretosol1} and \eqref{eqsemidissch}, respectively. Assume $G_{0}$, $f(0)\in L^{2}(\Omega)$, and $\int_{0}^{t}(t-s)^{\alpha-1}\|f'(s)\|_{L^{2}(\Omega)}ds<\infty$. Then one has
	\begin{equation*}
		\|G(t_{n})-G^{n}\|_{L^{2}(\Omega)}\leq C\tau\left (t_{n}^{-1}\ln(n)\|G_{0}\|_{L^{2}(\Omega)}+t_{n}^{\alpha-1}\|f(0)\|_{L^{2}(\Omega)}+\int_{0}^{t_{n}}(t_{n}-s)^{\alpha-1}\|f'(s)\|_{L^{2}(\Omega)}ds\right ).
	\end{equation*}
\end{theorem}
\begin{proof}
	According to \eqref{eqsollap2}, \eqref{eqtimschsol}, and the definition of $R(t)$, there exists 
	\begin{equation*}
		\begin{aligned}
			&\|G(t_{n})-G^{n}\|_{L^{2}(\Omega)}\\
			\leq&C\sum_{k=0}^{m}\Bigg \|\int_{\Gamma_{\theta,\kappa}}e^{zt_{n}}\left(\frac{1}{2}(z^{\alpha}+\mathcal{L})^{-1}\right )^{k}(z^{\alpha}+\mathcal{L})^{-1}z^{-1}dzf(0)\\
			&\quad-\int_{\Gamma^{\tau}_{\theta,\kappa}}e^{zt_{n-1}}\left(\frac{1}{2}((	\delta_{\tau}(e^{-z\tau}))^{\alpha}+\mathcal{L})^{-1}\right )^{k}((	\delta_{\tau}(e^{-z\tau}))^{\alpha}+\mathcal{L})^{-1}(\delta_{\tau}(e^{-z\tau}))^{-1}dz f(0)\Bigg \|_{L^{2}(\Omega)}\\
			&+C\sum_{k=0}^{m}\Bigg \|\int_{\Gamma_{\theta,\kappa}}e^{zt_{n}}\left(\frac{1}{2}(z^{\alpha}+\mathcal{L})^{-1}\right )^{k}(z^{\alpha}+\mathcal{L})^{-1}\tilde{R}dz\\
			&\quad-\int_{\Gamma^{\tau}_{\theta,\kappa}}e^{zt_{n}}\left(\frac{1}{2}((	\delta_{\tau}(e^{-z\tau}))^{\alpha}+\mathcal{L})^{-1}\right )^{k}((	\delta_{\tau}(e^{-z\tau}))^{\alpha}+\mathcal{L})^{-1}\tau\sum_{j=1}^{\infty}R(t_{j})e^{-zt_{j}}dz\Bigg \|_{L^{2}(\Omega)}\\
			&+C\sum_{k=0}^{m}\Bigg \|\int_{\Gamma_{\theta,\kappa}}e^{zt_{n}}\left(\frac{1}{2}(z^{\alpha}+\mathcal{L})^{-1}\right )^{k}(z^{\alpha}+\mathcal{L})^{-1}z^{\alpha-1}G^{0}dz\\
			&\quad-\int_{\Gamma^{\tau}_{\theta,\kappa}}e^{zt_{n-1}}\left(\frac{1}{2}((	\delta_{\tau}(e^{-z\tau}))^{\alpha}+\mathcal{L})^{-1}\right )^{k}((	\delta_{\tau}(e^{-z\tau}))^{\alpha}+\mathcal{L})^{-1}(	\delta_{\tau}(e^{-z\tau}))^{\alpha-1}G^{0}dz\Bigg \|_{L^{2}(\Omega)}\\
			&+C\Bigg \|\frac{1}{2\pi\mathbf{i}}\int_{\Gamma_{\theta,\kappa}}e^{zt_{n}}\left(\frac{1}{2}(z^{\alpha}
+\mathcal{L})^{-1}\right )^{m+1}\tilde{G}dz-\frac{1}{2\pi\mathbf{i}}\int_{\Gamma^{\tau}_{\theta,\kappa}}e^{zt_{n}}\left(\frac{1}{2}((	\delta_{\tau}(e^{-z\tau}))^{\alpha} +\mathcal{L})^{-1}\right )^{m+1}\tau\sum_{j=1}^{\infty}G^{j}e^{-zt_{j}}dz\Bigg \|_{L^{2}(\Omega)}\\			\leq&\sum_{k=0}^{m}\uppercase\expandafter{\romannumeral1}_{k}+\sum_{k=0}^{m}\uppercase\expandafter{\romannumeral2}_{k}+\sum_{k=0}^{m}\uppercase\expandafter{\romannumeral3}_{k}+\uppercase\expandafter{\romannumeral4}.
		\end{aligned}
	\end{equation*}
	For $\uppercase\expandafter{\romannumeral1}_{k}$, it can be split into three parts, i.e.,
	\begin{equation*}
		\begin{aligned}
			\uppercase\expandafter{\romannumeral1}_{k}\leq&C\Bigg \|\int_{\Gamma_{\theta,\kappa}\backslash\Gamma^{\tau}_{\theta,\kappa}}e^{zt_{n}}\left(\frac{1}{2}(z^{\alpha}+\mathcal{L})^{-1}\right )^{k}(z^{\alpha}+\mathcal{L})^{-1}z^{-1}dz\Bigg \|\|f(0)\|_{L^2(\Omega)}\\
			&+C\Bigg \|\int_{\Gamma^{\tau}_{\theta,\kappa}}(e^{zt_{n}}-e^{zt_{n-1}})\left(\frac{1}{2}(z^{\alpha}+\mathcal{L})^{-1}\right )^{k}(z^{\alpha}+\mathcal{L})^{-1}z^{-1}dz\Bigg\|\| f(0) \|_{L^{2}(\Omega)}\\
			&+C\Bigg \|\int_{\Gamma^{\tau}_{\theta,\kappa}}e^{zt_{n-1}}\Bigg(\left(\frac{1}{2}(z^{\alpha}+\mathcal{L})^{-1}\right )^{k}(z^{\alpha}+\mathcal{L})^{-1}z^{-1}\\
			-&\left(\frac{1}{2}((	\delta_{\tau}(e^{-z\tau}))^{\alpha}+\mathcal{L})^{-1}\right )^{k}((	\delta_{\tau}(e^{-z\tau}))^{\alpha}+\mathcal{L})^{-1}(\delta_{\tau}(e^{-z\tau}))^{-1}\Bigg)dz\Bigg\|\| f(0) \|_{L^{2}(\Omega)}\\
			\leq&\uppercase\expandafter{\romannumeral1}_{k,1}+\uppercase\expandafter{\romannumeral1}_{k,2}+\uppercase\expandafter{\romannumeral1}_{k,3}.
		\end{aligned}
	\end{equation*}
	Lemma \ref{lemresest} and simple calculations imply
	\begin{equation*}
		\begin{aligned}
			\uppercase\expandafter{\romannumeral1}_{k,1}\leq& C\tau\int_{\Gamma_{\theta,\kappa}\backslash\Gamma^{\tau}_{\theta,\kappa}}|e^{zt_{n}}|\left\|\frac{1}{2}(z^{\alpha}+\mathcal{L})^{-1}\right \|^{k}\|(z^{\alpha}+\mathcal{L})^{-1}\||dz|\|f(0)\|_{L^2(\Omega)}\\
			\leq& C\tau\int_{\Gamma_{\theta,\kappa}\backslash\Gamma^{\tau}_{\theta,\kappa}}|e^{zt_{n}}||z|^{-(k+1)\alpha}|dz|\|f(0)\|_{L^2(\Omega)}\\
			\leq&C\tau t_{n}^{\alpha-1}\|f(0)\|_{L^2(\Omega)}.
		\end{aligned}
	\end{equation*}
	According to the fact $|\frac{1-e^{-z\tau}}{\tau}|\leq C|z|$, one has
	\begin{equation*}
		\uppercase\expandafter{\romannumeral1}_{k,2}\leq C\tau t_{n}^{\alpha-1}\|f(0)\|_{L^2(\Omega)}.
	\end{equation*}
	Using the fact $|z-\delta_{\tau}(e^{-z\tau})|\leq C|z|^{2}\tau$ \citep[see, e.g.,][]{Jin.2016TFDSfFDaDWEwND,Jin.2017CoHOBCQfFEE,Lubich.1996Ndeefaoaeewaptmt} and doing simple calculations lead to
	\begin{equation*}
		\begin{aligned}
			&\Bigg\|\left(\frac{1}{2}(z^{\alpha}+\mathcal{L})^{-1}\right )^{k}(z^{\alpha}+\mathcal{L})^{-1}z^{-1}\\
			&\qquad-\left(\frac{1}{2}((	\delta_{\tau}(e^{-z\tau}))^{\alpha}+\mathcal{L})^{-1}\right )^{k}((	\delta_{\tau}(e^{-z\tau}))^{\alpha}+\mathcal{L})^{-1}(\delta_{\tau}(e^{-z\tau}))^{-1}\Bigg\|\leq C\tau|z|^{-(k+1)\alpha},
		\end{aligned}
	\end{equation*}
	which yields
	\begin{equation*}
		\uppercase\expandafter{\romannumeral1}_{k,3}\leq C\tau t_{n}^{\alpha-1}\|f(0)\|_{L^{2}(\Omega)}.
	\end{equation*}
	Similarly, one has
	\begin{equation*}
		\uppercase\expandafter{\romannumeral3}_{k}\leq C\tau t_{n}^{-1}\|G_{0}\|_{L^{2}(\Omega)}.
	\end{equation*}
	As for $\uppercase\expandafter{\romannumeral2}_{k}$, we split it into two parts, i.e.,
	\begin{equation*}
		\begin{aligned}
			\uppercase\expandafter{\romannumeral2}_{k}\leq&C\Bigg\|\int_{0}^{t_{n}}\Bigg(\int_{\Gamma_{\theta,\kappa}}e^{z(t_{n}-s)}\left(\frac{1}{2}(z^{\alpha}+\mathcal{L})^{-1}\right )^{k}(z^{\alpha}+\mathcal{L})^{-1}z^{-1}dz\\
			&\quad-\int_{\Gamma^{\tau}_{\theta,\kappa}}e^{z(t_{n}-s)}\left(\frac{1}{2}((	\delta_{\tau}(e^{-z\tau}))^{\alpha}+\mathcal{L})^{-1}\right )^{k}((	\delta_{\tau}(e^{-z\tau}))^{\alpha}+\mathcal{L})^{-1}(\delta_{\tau}(e^{-z\tau}))^{-1}dz\Bigg)f'(s)ds\Bigg\|_{L^2(\Omega)}\\
			&+C\Bigg\|\int_{0}^{t_{n}}\int_{\Gamma^{\tau}_{\theta,\kappa}}e^{z(t_{n}-s)}\left(\frac{1}{2}((	\delta_{\tau}(e^{-z\tau}))^{\alpha}+\mathcal{L})^{-1}\right )^{k}((	\delta_{\tau}(e^{-z\tau}))^{\alpha}+\mathcal{L})^{-1}(\delta_{\tau}(e^{-z\tau}))^{-1}dzf'(s)ds\\
			&\quad-\int_{\Gamma^{\tau}_{\theta,\kappa}}e^{zt_{n}}\left(\frac{1}{2}((	\delta_{\tau}(e^{-z\tau}))^{\alpha}+\mathcal{L})^{-1}\right )^{k}((	\delta_{\tau}(e^{-z\tau}))^{\alpha}+\mathcal{L})^{-1}\tau\sum_{j=1}^{\infty}R(t_{j})e^{-zt_{j}}dz\Bigg\|_{L^2(\Omega)}\\
			\leq&\uppercase\expandafter{\romannumeral2}_{k,1}+\uppercase\expandafter{\romannumeral2}_{k,2}.
		\end{aligned}
	\end{equation*}
	Similar to the derivation of $\uppercase\expandafter{\romannumeral1}_{k}$, one has
	\begin{equation*}
		\uppercase\expandafter{\romannumeral2}_{k,1}\leq C\tau\int_{0}^{t_{n}}(t_{n}-s)^{\alpha-1}\|f'(s)\|_{L^2(\Omega)}ds.
	\end{equation*}
	As for $\uppercase\expandafter{\romannumeral2}_{k,2}$, the following  fact
	\begin{equation*}
		\begin{aligned}
			\tau\sum_{n=1}^{\infty}R(t_{n})e^{-zt_{n}}=&\tau\sum_{n=1}^{\infty}\int_{0}^{t_{n}}f'(r)dre^{-zt_{n}}\\
			=&\tau\sum_{n=1}^{\infty}\sum_{j=1}^{n}\int_{t_{j-1}}^{t_{j}}f'(r)dre^{-zt_{n}}\\
			=&\frac{\tau}{1-e^{-z\tau}}\sum_{j=1}^{\infty}\left (e^{-zt_{j}}\int_{t_{j-1}}^{t_{j}}f'(r)dr\right )
		\end{aligned}
	\end{equation*}
	yields
	\begin{equation*}
		\begin{aligned}
			\uppercase\expandafter{\romannumeral2}_{k,2}\leq&C\Bigg\|\int_{0}^{t_{n}}\int_{\Gamma^{\tau}_{\theta,\kappa}}e^{z(t_{n}-s)}\left(\frac{1}{2}((	\delta_{\tau}(e^{-z\tau}))^{\alpha}+\mathcal{L})^{-1}\right )^{k}((	\delta_{\tau}(e^{-z\tau}))^{\alpha}+\mathcal{L})^{-1}(\delta_{\tau}(e^{-z\tau}))^{-1}dzf'(s)ds\\
			&\quad-\sum_{j=1}^{n}\int_{t_{j-1}}^{t_{j}}\int_{\Gamma^{\tau}_{\theta,\kappa}}e^{z(t_{n}-t_{j})}\left(\frac{1}{2}((	\delta_{\tau}(e^{-z\tau}))^{\alpha}+\mathcal{L})^{-1}\right )^{k}((	\delta_{\tau}(e^{-z\tau}))^{\alpha}+\mathcal{L})^{-1}(\delta_{\tau}(e^{-z\tau}))^{-1}dzf'(s)ds\Bigg\|_{L^2(\Omega)}\\
			\leq&C\tau\int_{0}^{t_{n}}(t_{n}-s)^{\alpha-1}\|f'(s)\|_{L^2(\Omega)}ds.
		\end{aligned}
	\end{equation*}
	As for $\uppercase\expandafter{\romannumeral4}$, we introduce $\epsilon_{k}=G(t_{k})-G^{k}$ and define $B_{j}$ by
	\begin{equation*}
		B_{j}=\tau\int_{\Gamma^{\tau}_{\theta,\kappa}}e^{zt_{j}}\left(\frac{1}{2}((	\delta_{\tau}(e^{-z\tau}))^{\alpha}+\mathcal{L})^{-1}\right )^{m+1}dz.
	\end{equation*}
	Here we choose $m=\lfloor\frac{1}{\alpha}\rfloor$. It is easy to verify
	\begin{equation}\label{eqpropBj}
		\begin{aligned}
			\|B_{j}\|\leq C\tau t_{j}^{(m+1)\alpha-1}.
		\end{aligned}
	\end{equation}
	Thus, we can rewrite $\uppercase\expandafter{\romannumeral4}$ as
	\begin{equation*}
		\begin{aligned}
			\uppercase\expandafter{\romannumeral4}\leq&C\sum_{j=1}^{n}\Bigg \|\int_{t_{j-1}}^{t_{j}}\int_{\Gamma_{\theta,\kappa}}e^{z(t_{n}-s)}\left(\frac{1}{2}(z^{\alpha}+\mathcal{L})^{-1}\right )^{m+1}dzG(s)ds-B_{n-j}G^{j}\Bigg \|_{L^{2}(\Omega)}\\
			\leq&C\sum_{j=1}^{n}\Bigg \|\int_{t_{j-1}}^{t_{j}}\int_{\Gamma_{\theta,\kappa}}e^{z(t_{n}-s)}\left(\frac{1}{2}(z^{\alpha}+\mathcal{L})^{-1}\right )^{m+1}dz(G(s)-G(t_{j}))ds\Bigg \|_{L^{2}(\Omega)}\\
			&+C\sum_{j=1}^{n}\Bigg \|\left (\int_{t_{j-1}}^{t_{j}}\int_{\Gamma_{\theta,\kappa}}e^{z(t_{n}-s)}\left(\frac{1}{2}(z^{\alpha}+\mathcal{L})^{-1}\right )^{m+1}dzds-B_{n-j}\right )G(t_{j})\Bigg \|_{L^{2}(\Omega)}\\
			&+C\sum_{j=1}^{n}\Bigg \|B_{n-j}\epsilon_{j}\Bigg \|_{L^{2}(\Omega)}\\
			\leq& \uppercase\expandafter{\romannumeral4}_{1}+\uppercase\expandafter{\romannumeral4}_{2}+\uppercase\expandafter{\romannumeral4}_{3}.
		\end{aligned}
	\end{equation*}
	Using Theorem \ref{thmholdere}, one can get
	\begin{equation*}
		\begin{aligned}
			\uppercase\expandafter{\romannumeral4}_{1}\leq& C\Bigg \|\int_{t_{0}}^{t_{1}}\int_{\Gamma_{\theta,\kappa}}e^{z(t_{n}-s)}\left(\frac{1}{2}(z^{\alpha}+\mathcal{L})^{-1}\right )^{m+1}dz(G(s)-G(t_{1}))ds\Bigg \|_{L^{2}(\Omega)}\\
			&+ \sum_{j=2}^{n}C\Bigg \|\int_{t_{j-1}}^{t_{j}}\int_{\Gamma_{\theta,\kappa}}e^{z(t_{n}-s)}\left(\frac{1}{2}(z^{\alpha}+\mathcal{L})^{-1}\right )^{m+1}dz(G(s)-G(t_{j}))ds\Bigg \|_{L^{2}(\Omega)}\\
			\leq&C\int_{t_{0}}^{t_{1}}\int_{\Gamma_{\theta,\kappa}}|e^{z(t_{n}-s)}||z|^{-(m+1)\alpha}|dz|(\|G(s)\|_{L^2(\Omega)}+\|G(t_{1})\|_{L^{2}(\Omega)})ds\\
			&+ \sum_{j=2}^{n}C \int_{t_{j-1}}^{t_{j}}\int_{\Gamma_{\theta,\kappa}}|e^{z(t_{n}-s)}||z|^{-(m+1)\alpha}|dz|\int_{s}^{t_{j}}\|G'(r)\|_{L^{2}(\Omega)}drds	\\
			\leq&C\tau(\|G(s)\|_{L^2(\Omega)}+\|G(t_{1})\|_{L^{2}(\Omega)})\\
			&+\sum_{j=2}^{n}C\int_{t_{j-1}}^{t_{j}}(t_{n}-s)^{(m+1)\alpha-1}\int_{s}^{t_{j}}\|G'(r)\|_{L^{2}(\Omega)}drds\\
			\leq&C\tau\left (\ln(n)\|G_{0}\|_{L^{2}(\Omega)}+\|f(0)\|_{L^{2}(\Omega)}+\int_{0}^{t_{n}}\|f'(s)\|_{L^{2}(\Omega)}ds\right ).
		\end{aligned}
	\end{equation*}
	Simple calculations lead to
	\begin{equation*}
		\begin{aligned}
			&\uppercase\expandafter{\romannumeral4}_{2}\\
			\leq&C\sum_{j=1}^{n}\Bigg \|\left (\int_{t_{j-1}}^{t_{j}}\int_{\Gamma_{\theta,\kappa}}e^{z(t_{n}-s)}\left(\frac{1}{2}(z^{\alpha}+\mathcal{L})^{-1}\right )^{m+1}dzds\right.\\
			&\qquad\qquad\left.-\tau\int_{\Gamma^{\tau}_{\theta,\kappa}}e^{z(t_{n}-t_{j})}\left(\frac{1}{2}((	\delta_{\tau}(e^{-z\tau}))^{\alpha}+\mathcal{L})^{-1}\right )^{m+1}dz\right )\Bigg \|\|G(t_{j}) \|_{L^{2}(\Omega)}\\
			\leq&C\sum_{j=1}^{n}\Bigg \|\int_{t_{j-1}}^{t_{j}}\int_{\Gamma_{\theta,\kappa}^{\tau}}e^{z(t_{n}-t_{j})}\left(\left(\frac{1}{2}(z^{\alpha}+\mathcal{L})^{-1}\right )^{m+1}-\left(\frac{1}{2}((	\delta_{\tau}(e^{-z\tau}))^{\alpha}+\mathcal{L})^{-1}\right )^{m+1}\right )dzds\Bigg \|\|G(t_{j}) \|_{L^{2}(\Omega)}\\
			&+C\sum_{j=1}^{n}\Bigg \|\int_{t_{j-1}}^{t_{j}}\int_{\Gamma_{\theta,\kappa}^{\tau}}(e^{z(t_{n}-s)}-e^{z(t_{n}-t_{j})})\left(\frac{1}{2}(z^{\alpha}+\mathcal{L})^{-1}\right )^{m+1}dzds\Bigg \|\|G(t_{j}) \|_{L^{2}(\Omega)}\\
			&+C\sum_{j=1}^{n}\Bigg \|\int_{t_{j-1}}^{t_{j}}\int_{\Gamma_{\theta,\kappa}\backslash\Gamma^{\tau}_{\theta,\kappa}}e^{z(t_{n}-s)}\left(\frac{1}{2}(z^{\alpha}+\mathcal{L})^{-1}\right )^{m+1}dzds\Bigg \|\|G(t_{j}) \|_{L^{2}(\Omega)}\\
			\leq&C\tau^{2}\sum_{j=1}^{n}(t_{n}-t_{j}+\tau)^{(m+1)\alpha-2}\|G(t_{j}) \|_{L^{2}(\Omega)}+C\tau\sum_{j=1}^{n}\int_{t_{j-1}}^{t_{j}}(t_{n}-s)^{(m+1)\alpha-2}ds\|G(t_{j}) \|_{L^{2}(\Omega)}\\
			\leq&C\tau\left (\|G_{0}\|_{L^{2}(\Omega)}+\|f(0)\|_{L^{2}(\Omega)}+\int_{0}^{t_{n}}\|f'(s)\|_{L^{2}(\Omega)}ds\right ).
		\end{aligned}
	\end{equation*}
	As for $\uppercase\expandafter{\romannumeral4}_{3}$, using \eqref{eqpropBj}, one has
	\begin{equation*}
		\uppercase\expandafter{\romannumeral4}_{3}\leq C\tau\sum_{j=1}^{n}(t_{n}-t_{j})^{(m+1)\alpha-1}\|\epsilon_{j}\|_{L^2(\Omega)}.
	\end{equation*}
	Thus combining the discrete Gr\"onwall's inequality \citep[see, e.g.,][]{Thomee.2006Gfemfpp}, the desired result is reached.
\end{proof}

\section{Fully-discrete scheme and error analysis}
In this section, we provide the fully-discrete scheme by using local discontinuous Galerkin (LDG) method to discretize the operator $\mathcal{L}$ and discuss the resulting error estimates.

Introduce a well approximation of the physical domain $\Omega$ by the computational domain $\Omega_{h}$ and the rectangle meshes are used. Let mesh size $h=1/N$ $(N\in\mathbb{N}^{*})$, $x_{i}=v_{i}=ih$, $(i=0,1,\ldots,N)$ and the elements
\begin{equation*}
	I_{i,j}=(x_{i-1},x_{i})\times(v_{j-1},v_{j})\quad i,j=1,2,\ldots,N.
\end{equation*} 
The polynomial space $\mathbb{P}_{k}(I_{i,j})$ consists of polynomials in $I_{i,j}$ of degree at most $k$ $(k\geq 1)$ and the discontinuous finite element space $V_{h,k}$ can be defined by
\begin{equation*}
	V_{h,k}=\{v:\Omega_{h}\rightarrow \mathbb{R}|~v|_{I_{i,j}}\in \mathbb{P}_{k}(I_{i,j}),i,j=1,\ldots,N\}.
\end{equation*}

Let $\mathbf{P}^{n}=\{P^{n}_{x},P^{n}_{v}\}=\nabla G^{n}$. According to \eqref{eqsemidissch} and the definition of $\mathcal{L}$, we have
\begin{equation}\label{eqvarform}
	\left \{\begin{aligned}
		&\left(\sum_{k=0}^{n-1}d^{(\alpha)}_{k}(G^{n-k}-G^{0}),\mu\right)_{I_{i,j}}+\left(vP^{n}_{x},\mu\right)_{I_{i,j}}-\left(vP^{n}_{v},\mu\right)_{I_{i,j}}\\
		&\qquad\qquad\qquad+\left (P^{n}_{v},\frac{\partial}{\partial v}\mu\right )_{I_{i,j}}-\int_{x_{i-1}}^{x_{i}}P^{n}_{v}\mu|_{v=v_{j-1}}^{v=v_{j}}dx=\left(f^{n},\mu\right)_{I_{i,j}}+\left(G^{n},\mu\right)_{I_{i,j}},\\
		&\left(\mathbf{P}^{n},\boldsymbol{\nu}\right)_{I_{i,j}}=-(G^{n},\nabla\cdot \boldsymbol{\nu})_{I_{i,j}}+(G^{n},\mathbf{n}\cdot\boldsymbol{\nu})_{\partial I_{i,j}},
		\\
		&G^{n}(x,0)=G^{n}(x,1)=G^{n}(0,v)=0,\\
		&(G^{0},\mu)_{I_{i,j}}=(G_{0},\mu)_{I_{i,j}},
	\end{aligned}\right .
\end{equation}
for all $\mu\in H^{1}(\Omega)$ and $\boldsymbol{\nu}\in (H^{1}(\Omega))^{2}$. Here $\mathbf{n}$ means the outward unit vector of $\partial I_{i,j}$.  Let $\{G^{n}_{h},\mathbf{P}^{n}_{h}\}=\{G^{n}_{h},\{P^{n}_{x,h},P^{n}_{v,h}\}\}$ be the approximation of $\{G^{n},\mathbf{P}^{n}\}$. Then we can write the LDG scheme as: Find $\{G^{n}_{h},\mathbf{P}^{n}_{h}\}\in V_{h,k}\times (V_{h,k})^{2}$ such that
\begin{equation}\label{eqfullscheme}
	\left \{\begin{aligned}
		&\left(\sum_{k=0}^{n-1}d^{(\alpha)}_{k}(G^{n-k}_{h}-G^{0}_{h}),\mu_{h}\right)_{I_{i,j}}+\left(vP^{n}_{x,h},\mu_{h}\right)_{I_{i,j}}-\left(vP^{n}_{v,h},\mu_{h}\right)_{I_{i,j}}\\
		&\qquad\qquad\qquad-\left (\frac{\partial}{\partial v}P^{n}_{v,h},\mu_{h}\right )_{I_{i,j}}+\int_{x_{i-1}}^{x_{i}}(P^{n}_{v,h}-\hat{P}^{n}_{v,h})\mu_{h}|_{v=v_{j-1}}^{v=v_{j}}dx=\left(f^{n},\mu_{h}\right)_{I_{i,j}}+\left(G^{n}_{h},\mu_{h}\right)_{I_{i,j}},\\
		&\left(\mathbf{P}^{n}_{h},\boldsymbol{\nu}_{h}\right)_{I_{i,j}}=(\nabla G^{n}_{h},\boldsymbol{\nu}_{h})_{I_{i,j}}-(G^{n}_{h}-\hat{G}^{n}_{h},\mathbf{n}\cdot\boldsymbol{\nu}_{h})_{\partial I_{i,j}},
		\\
		&G^{n}(x,0)=G^{n}(x,1)=G^{n}(0,v)=0,\\
		&(G^{0}_{h},\mu_{h})_{I_{i,j}}=(G_{0},\mu_{h})_{I_{i,j}},
	\end{aligned}\right .
\end{equation}
for all $\mu_{h}\in V_{h,k}$ and $\boldsymbol{\nu}_{h}\in (V_{h,k})^{2}$. Here we choose the fluxes \citep[see, e.g.,][]{Castillo.2001OapeefthvotldGmfcdp,Cockburn.1998TLDGMfTDCDS} 
\begin{equation}\label{eqfluxx}
	\begin{aligned}
		&\hat{G}^{n}_{h}(x_{i},v)=G^{n}_{h}(x_{i}^{-},v)=\lim_{x\rightarrow x_{i}^{-} }G^{n}_{h}(x,v)\qquad i=1,\ldots, N,\\
		&\hat{G}^{n}_{h}(x_{0},v)=G^{n}(x_{0},v),\\
	\end{aligned}
\end{equation}
and
\begin{equation}\label{eqfluxv2}
	\begin{aligned}
		&\hat{G}^{n}_{h}(x,v_{j})=G^{n}_{h}(x,v_{j}^{+})=\lim_{v\rightarrow v_{j}^{+} }G^{n}_{h}(x,v)\quad j=1,\ldots,N-1,\\
		&\hat{G}^{n}_{h}(x,v_{0})=\hat{G}^{n}_{h}(x,v_{N})=0,\\ &\hat{P}^{n}_{v,h}(x,v_{j})=P^{n}_{v,h}(x,v_{j}^{-})=\lim_{v\rightarrow v_{j}^{-} }P^{n}_{v,h}(x,v)\quad j=1,\ldots,N,\\
		&\hat{P}^{n}_{v,h}(x,v_{0})=P^{n}_{v,h}(x,v_{0}^{+})+\frac{\vartheta G^{n}_{h}(x,v_{0}^{+})}{h},
	\end{aligned}
\end{equation}
where $\vartheta$ is a positive constant.

Denote $\boldsymbol{\phi}_{h}=\{G^{k}_{h},\mathbf{P}^{k}_{h}\}_{k=1}^{\infty}$ and $\boldsymbol{\psi}_{h}=\{\mu_{h},\boldsymbol{\nu}_{h},\eta_{h}\}$. Introduce 
\begin{equation*}
	\begin{aligned}
		\mathcal{B}_{n}(\boldsymbol{\phi}_{h},\boldsymbol{\psi}_{h})=&\sum_{i,j=1}^{N}\Bigg(\left(\sum_{k=0}^{n-1}d^{(\alpha)}_{k}(G^{n-k}_{h}-G^{0}_{h}),\mu_{h}\right)_{I_{i,j}}+\left(vP^{n}_{x,h},\mu_{h}\right)_{I_{i,j}}-\left(vP^{n}_{v,h},\mu_{h}\right)_{I_{i,j}}\\
		&\qquad\qquad-\left (\frac{\partial}{\partial v}P^{n}_{v,h},\mu_{h}\right )_{I_{i,j}}+\int_{x_{i-1}}^{x_{i}}(P^{n}_{v,h}-\hat{P}^{n}_{v,h})\mu_{h}|_{v=v_{j-1}}^{v=v_{j}}dx-\left(G^{n}_{h},\mu_{h}\right)_{I_{i,j}}\\
		&\qquad\qquad-\left(\mathbf{P}^{n}_{h},\boldsymbol{\nu}_{h}\right)_{I_{i,j}}+(\nabla G^{n}_{h},\boldsymbol{\nu}_{h})_{I_{i,j}}-(G^{n}_{h}-\hat{G}^{n}_{h},\mathbf{n}\cdot\boldsymbol{\nu}_{h})_{\partial I_{i,j}}\\
		&\qquad\qquad+(P^{n}_{v,h},\eta_{h})_{I_{i,j}}-(\frac{\partial}{\partial v} G^{n}_{h},\eta_{h})_{I_{i,j}}+\int_{x_{i-1}}^{x_{i}}(G^{n}_{h}-\hat{G}^{n}_{h})\eta_{h}|_{v=v_{j-1}}^{v=v_{j}}dx\Bigg).
	\end{aligned}
\end{equation*}
Thus we can rewrite \eqref{eqfullscheme} as
\begin{equation*}
	\mathcal{B}_{n}(\boldsymbol{\phi}_{h},\boldsymbol{\psi}_{h})=\sum_{i,j=1}^{N}(f^{n},\mu_{h})_{I_{i,j}}.
\end{equation*}

In what follows, we consider stability of the provided scheme. Let $\{\bar{G}^{n}_{h},\bar{\mathbf{P}}^{n}_{h}\}$ be the perturbed solution
of $\{G^{n}_{h},\mathbf{P}^{n}_{h}\}$. Denote $\boldsymbol{\epsilon}=\{\boldsymbol{\epsilon}^{k}\}_{k=1}^{\infty}$ and $\boldsymbol{\epsilon}^{n}=\{\epsilon^{n}_{G},\boldsymbol{\epsilon}^{n}_{\mathbf{P}}\}$ with $\epsilon^{n}_{G}=G^{n}_{h}-\bar{G}^{n}_{h}$, $\boldsymbol{\epsilon}^{n}_{\mathbf{P}}=\{\epsilon^{n}_{P,x},\epsilon^{n}_{P,v}\}=\mathbf{P}^{n}_{h}-\bar{\mathbf{P}}^{n}_{h}$. Then there holds
\begin{equation*}
	\mathcal{B}_{n}(\boldsymbol{\epsilon},\boldsymbol{\psi}_{h})=0.
\end{equation*}
To prove stability of the scheme, we first provide a new discrete Gr\"onwall's inequality.
\begin{lemma}\label{lemstab}
	Let $u^{n}\geq 0$ $(n=0,1,\ldots)$ satisfy
	\begin{equation}\label{eqlemstabcon}
		\sum_{k=0}^{n-1}d^{(\alpha)}_{k}(u^{n-k}-u^{0})\leq u^{n}.
	\end{equation}
	Then we have
	\begin{equation*}
		u^{n}\leq C|u^{0}|.
	\end{equation*}
\end{lemma}

\begin{proof}
	Assume that $v^{n}$ satisfies
	\begin{equation}\label{eqlemstabcon2}
		\left \{
		\begin{aligned}
			&\sum_{k=0}^{n-1}d^{(\alpha)}_{k}(v^{n-k}-v^{0})= u^{n},\\
			&v^{0}=u^{0}.
		\end{aligned}
		\right .
	\end{equation}
	Multiplying $\zeta^{n}$ on both sides of \eqref{eqlemstabcon2} and summing $n$ from $1$ to $\infty$, we obtain
	\begin{equation*}
		\sum_{n=1}^{\infty}\sum_{k=0}^{n-1}d^{(\alpha)}_{k}(v^{n-k}-v^{0})\zeta^{n}= \sum_{n=1}^{\infty}u^{n}\zeta^{n}.
	\end{equation*}
	Combining the definition of $d^{(\alpha)}_{k}$, i.e., \eqref{eqdefdja}, one can get
	\begin{equation*}
		\sum_{n=1}^{\infty}v^{n}\zeta^{n}= \frac{1}{\tau}(\delta_{\tau}(\zeta))^{-1}\zeta u^{0}+(\delta_{\tau}(\zeta))^{-\alpha}\sum_{n=1}^{\infty}u^{n}\zeta^{n}.
	\end{equation*}
	Using Cauchy's integral theorem and doing simple calculations lead to
	\begin{equation*}
		\begin{aligned}
			v^{n}=&\frac{1}{2\pi\mathbf{i}}\int_{\Gamma^{\tau}_{\theta,\kappa}}e^{zt_{n-1}}(\delta_{\tau}(e^{-z\tau}))^{-1}u^{0}dz+\frac{\tau}{2\pi\mathbf{i}}\int_{\Gamma^{\tau}_{\theta,\kappa}}e^{zt_{n}}(\delta_{\tau}(e^{-z\tau}))^{-\alpha}\sum_{i=1}^{\infty}u^{i}e^{-zt_{i}}dz.\\
		\end{aligned}
	\end{equation*}
	Thus by using $C_{1}|z|\leq |\delta_{\tau}(e^{-z\tau})|\leq C_{2}|z|$ with $C_{1},~C_{2}$ being two positive constants for $z\in \Gamma^{\tau}_{\theta,\kappa}$ \citep[see, e.g.,][]{Jin.2016TFDSfFDaDWEwND,Jin.2017CoHOBCQfFEE} and simple calculations, there holds
	\begin{equation*}
		|v^{n}|\leq C|u^{0}|+C\tau\sum_{i=0}^{n-1}t_{i+1}^{\alpha-1}|u^{n-i}|.
	\end{equation*}
	Subtracting \eqref{eqlemstabcon2} from \eqref{eqlemstabcon}, we have
	\begin{equation*}
		\sum_{k=0}^{n-1}d^{(\alpha)}_{k}(u^{n-k}-v^{n-k})\leq 0.
	\end{equation*}
	Further, Lemma \ref{lempropb} gives
	\begin{equation*}
		u^{n}\leq v^{n}\leq C|u^{0}|+C\tau\sum_{i=0}^{n-1}t_{i+1}^{\alpha-1}|u^{n-i}|.
	\end{equation*}
	Combining the discrete Gr\"onwall inequality \citep[see, e.g.,][]{Thomee.2006Gfemfpp}, we have
	\begin{equation*}
		u^{n}\leq C|u^{0}|.
	\end{equation*}
	This completes the proof.
\end{proof}

Thanks to above Lemma \ref{lemstab}, we can get the following stability result.
\begin{theorem}\label{thmstab}
	The scheme \eqref{eqfullscheme} with fluxes \eqref{eqfluxx} and \eqref{eqfluxv2} is $L^{2}$ stable, i.e.,
	\begin{equation*}
		\|\epsilon^{n}_{G}\|_{L^{2}(\Omega)}\leq C\|\epsilon^{0}_{G}\|_{L^{2}(\Omega)},
	\end{equation*}
	where $\epsilon^{n}_{G}=G^{n}_{h}-\bar{G}^{n}_{h}$.
\end{theorem}
\begin{proof}
	Choosing $\mu_{h}=\epsilon^{n}_{G}$, $\boldsymbol{\nu}_{h}=\{v\epsilon^{n}_{G},-v\epsilon^{n}_{G}\}$, and $\eta_{h}=\epsilon^{n}_{P,v}$, one has
	\begin{equation*}
		\begin{aligned}
			\mathcal{B}_{n}(\boldsymbol{\epsilon},\boldsymbol{\psi}_{h})=&\sum_{i,j=1}^{N}\Bigg(\left(\sum_{k=0}^{n-1}d^{(\alpha)}_{k}(\epsilon^{n-k}_{G}-\epsilon^{0}_{G}),\epsilon^{n}_{G}\right)_{I_{i,j}}+\left(v\epsilon^{n}_{P,x},\epsilon^{n}_{G}\right)_{I_{i,j}}-\left(v\epsilon^{n}_{P,v},\epsilon^{n}_{G}\right)_{I_{i,j}}\\
			&\qquad\qquad-\left (\frac{\partial}{\partial v}\epsilon^{n}_{P,v},\epsilon^{n}_{G}\right )_{I_{i,j}}+\int_{x_{i-1}}^{x_{i}}(\epsilon^{n}_{P,v}-\hat{\epsilon}^{n}_{P,v})\epsilon^{n}_{G}|_{v=v_{j-1}}^{v=v_{j}}dx-\left(\epsilon^{n}_{G},\epsilon^{n}_{G}\right)_{I_{i,j}}\\
			&\qquad\qquad-\left(\boldsymbol{\epsilon}^{n}_{\mathbf{P}},\boldsymbol{\nu}_{h}\right)_{I_{i,j}}+(\nabla \epsilon^{n}_{G},\boldsymbol{\nu}_{h})_{I_{i,j}}-(\epsilon^{n}_{G}-\hat{\epsilon}^{n}_{G},\mathbf{n}\cdot\boldsymbol{\nu}_{h})_{\partial I_{i,j}}\\
			&\qquad\qquad+(\epsilon^{n}_{P,v},\epsilon^{n}_{P,v})_{I_{i,j}}-(\frac{\partial}{\partial v} \epsilon^{n}_{G},\epsilon^{n}_{P,v})_{I_{i,j}}+\int_{x_{i-1}}^{x_{i}}(\epsilon^{n}_{G}-\hat{\epsilon}^{n}_{G})\epsilon^{n}_{P,v}|_{v=v_{j-1}}^{v=v_{j}}dx\Bigg)=0.
		\end{aligned}
	\end{equation*}
	
	By the fluxes \eqref{eqfluxx} and \eqref{eqfluxv2}, we have
	\begin{equation*}
		\begin{aligned}
			&\sum_{i,j=1}^{N}\Bigg(-\left (\frac{\partial}{\partial v}\epsilon^{n}_{P,v},\epsilon^{n}_{G}\right )_{I_{i,j}}+\int_{x_{i-1}}^{x_{i}}(\epsilon^{n}_{P,v}-\hat{\epsilon}^{n}_{P,v})\epsilon^{n}_{G}|_{v=v_{j-1}}^{v=v_{j}}dx\\
			&\qquad\qquad\qquad-\left (\frac{\partial}{\partial v} \epsilon^{n}_{G},\epsilon^{n}_{P,v}\right )_{I_{i,j}}+\int_{x_{i-1}}^{x_{i}}(\epsilon^{n}_{G}-\hat{\epsilon}^{n}_{G})\epsilon^{n}_{P,v}|_{v=v_{j-1}}^{v=v_{j}}dx\Bigg)=\sum_{i=1}^{N}\int_{x_{i-1}}^{x_{i}}\frac{\vartheta (\epsilon^{n}_{G}(v_{0}^{+}))^{2}}{h}dx
		\end{aligned}
	\end{equation*}
	and
	\begin{equation*}
		\begin{aligned}
			&\sum_{i,j=1}^{N}\Bigg((\nabla \epsilon^{n}_{G},\boldsymbol{\nu}_{h})_{I_{i,j}}-(\epsilon^{n}_{G}-\hat{\epsilon}^{n}_{G},\mathbf{n}\cdot\boldsymbol{\nu}_{h})_{\partial I_{i,j}}\Bigg)\\
			&\qquad=\frac{1}{2}\sum_{j=1}^{N}\int_{v_{j-1}}^{v_{j}}\left(\sum_{i=1}^{N-1}v(\epsilon^{n}_{G}(x_{i}^{-})-\epsilon^{n}_{G}(x_{i}^{+}))^{2}+v(\epsilon^{n}_{G}(x_{0}^{+}))^{2}+v(\epsilon^{n}_{G}(x_{N}^{-}))^{2}\right)dv\\
			&\qquad~~+\frac{1}{2}\sum_{i=1}^{N}\int_{x_{i-1}}^{x_{i}}\left(\sum_{j=1}^{N-1}v_{j}(\epsilon^{n}_{G}(v_{j}^{-})-\epsilon^{n}_{G}(v_{j}^{+}))^{2}+v_{N}(\epsilon^{n}_{G}(v_{N}^{+}))^{2}\right)dx\\
			&\qquad~~+\frac{1}{2}\sum_{i,j=1}^{N}(\epsilon^{n}_{G},\epsilon^{n}_{G})_{I_{i,j}}.
		\end{aligned}
	\end{equation*}
	Thus
	\begin{equation*}
		\begin{aligned}
			&\sum_{i,j=1}^{N}\left(d^{(\alpha)}_{0}\epsilon^{n}_{G},\epsilon^{n}_{G}\right)_{I_{i,j}}\\
			\leq&\sum_{i,j=1}^{N}\Bigg(\left(\sum_{k=0}^{n-1}d^{(\alpha)}_{k}\epsilon^{0}_{G},\epsilon^{n}_{G}\right)_{I_{i,j}}-\left(\sum_{k=1}^{n-1}d^{(\alpha)}_{k}\epsilon^{n-k}_{G},\epsilon^{n}_{G}\right)_{I_{i,j}}+\frac{1}{2}\left(\epsilon^{n}_{G},\epsilon^{n}_{G}\right)_{I_{i,j}}\Bigg)\\
			\leq&\frac{1}{2}\sum_{i,j=1}^{N}\Bigg(\sum_{k=0}^{n-1}d^{(\alpha)}_{k}\|\epsilon^{0}_{G}\|_{L^{2}(I_{i,j})}^{2}+d^{(\alpha)}_{0}\|\epsilon^{n}_{G}\|_{L^{2}(I_{i,j})}^{2}-\sum_{k=1}^{n-1}d^{(\alpha)}_{k}\|\epsilon^{n-k}_{G}\|_{L^{2}(I_{i,j})}^{2}+\|\epsilon^{n}_{G}\|_{L^{2}(I_{i,j})}^{2}\Bigg),
		\end{aligned}
	\end{equation*}
	which leads to
	\begin{equation*}
		\begin{aligned}
			\sum_{k=0}^{n-1}d^{(\alpha)}_{k}\left (\sum_{i,j=1}^{N}\|\epsilon^{n-k}_{G}\|_{L^{2}(I_{i,j})}^{2}-\sum_{i,j=1}^{N}\|\epsilon^{0}_{G}\|_{L^{2}(I_{i,j})}^{2}\right )\leq\sum_{i,j=1}^{N}\|\epsilon^{n}_{G}\|_{L^{2}(I_{i,j})}^{2}.
		\end{aligned}
	\end{equation*}
Further combining Lemma \ref{lemstab}, we can obtain the desired result.
\end{proof}

At last, we provide the error estimate for the fully-discrete scheme \eqref{eqfullscheme}.
Introduce $L^{2}$ projection operator $\mathcal{P}_{x}$ in one dimension as, for all $(x_{i-1},x_{i})\in(0,1)$
\begin{equation*}
	\int_{x_{i-1}}^{x_{i}}(\mathcal{P}_{x}\mu-\mu)\nu dx=0\quad \forall \nu\in \mathbb{P}_{k}(x_{i-1},x_{i}).
\end{equation*}
Define the projection operators $\mathcal{P}^{+}_{x}$ and $\mathcal{P}^{-}_{x}$ as
\begin{equation*}
	\begin{aligned}
		&\int_{x_{i-1}}^{x_{i}}(\mathcal{P}^{+}_{x}\mu-\mu)\nu dx=0~~ \forall \nu\in \mathbb{P}_{k-1}(x_{i-1},x_{i})~~{\rm and}~~\mathcal{P}^{+}_{x}\mu(x_{i-1})=\mu(x_{i-1}),
	\end{aligned}
\end{equation*}
and
\begin{equation*}
	\begin{aligned}
		&\int_{x_{i-1}}^{x_{i}}(\mathcal{P}^{-}_{x}\mu-\mu)\nu dx=0~~\forall \nu\in \mathbb{P}_{k-1}(x_{i-1},x_{i})~~{\rm and}~~\mathcal{P}^{-}_{x}\mu(x_{i})=\mu(x_{i}).
	\end{aligned}
\end{equation*}
Moreover, define the following projections in two dimensions by tensor products
\begin{equation*}
	\begin{aligned}
		& \Pi=\mathcal{P}_{x}^{-}\otimes\mathcal{P}_{v}^{+},\quad\Pi_{x}=\mathcal{P}_{x}\otimes\mathcal{P}_{v},\quad \Pi^{-}_{v}=\mathcal{P}_{x}\otimes\mathcal{P}^{-}_{v}.
	\end{aligned}
\end{equation*}

\begin{lemma}\citep[see, e.g.,][]{Cockburn.2001SotLDGMfEPoCG}\label{lemspec}
	For any $u\in \mathbb{P}_{k+1}(I_{i,j})$ and $\boldsymbol{\nu}=\{\nu_{x},\nu_{v}\}\in (\mathbb{P}_{k}(I_{i,j}))^{2}$, we have
	\begin{equation*}
		\begin{aligned}
			&\left(u-\Pi u,\frac{\partial}{\partial x}\nu_{x}\right)_{I_{i,j}}-\int_{v_{j-1}}^{v_{j}}(u-\widehat{\Pi u})\nu_{x}|^{x_{i}}_{x_{i-1}}dv=0,\\
			&\left(u-\Pi u,\frac{\partial}{\partial v}\nu_{v}\right)_{I_{i,j}}-\int_{x_{i-1}}^{x_{i}}(u-\widehat{\Pi u})\nu_{v}|^{v_{j}}_{v_{j-1}}dx=0,
		\end{aligned}
	\end{equation*}
	and
	\begin{equation*}
		(u-\Pi u,\nabla\cdot\boldsymbol{\nu})_{I_{i,j}}-(u-\widehat{\Pi u},\mathbf{n}\cdot\boldsymbol{\nu})_{\partial I_{i,j}}=0.
	\end{equation*}
	
\end{lemma}

\begin{theorem}\label{thmfullerr}
	Let $G^{n}$ and $G^{n}_{h}$ be the solutions of \eqref{eqsemidissch} and \eqref{eqfullscheme}, respectively. If $G^{i}\in H^{k+2}(\Omega)$ and $\mathbf{P}^{i}\in(H^{k+1}(\Omega))^{2}$ for $i=0,1,2,\ldots,n$, then we have
	\begin{equation*}
		\|G^{n}-G^{n}_{h}\|_{L^2(\Omega)}\leq Ch^{k+1}.
	\end{equation*}
\end{theorem}
\begin{proof}
	Introduce
	\begin{equation*}
		\mathbf{e}=\{e^{k}_{G},\mathbf{e}^{k}_{\mathbf{P}}\}_{k=1}^{\infty}=\{G^{k}-G^{k}_{h},\mathbf{P}^{k}-\mathbf{P}^{k}_{h}\}_{k=1}^{\infty}.
	\end{equation*}
	From \eqref{eqvarform} and \eqref{eqfullscheme}, it's easy to verify
	\begin{equation*}
		\mathcal{B}_{n}(\mathbf{e},\boldsymbol{\psi}_{h})=0
	\end{equation*}
	for $\boldsymbol{\psi}_{h}\in V_{h,k}\times (V_{h,k})^{2}\times V_{h,k}$. Taking $\boldsymbol{\psi}_{h}=\{\mu_{h},\boldsymbol{\nu}_{h},\eta_{h}\}=\{\Pi e^{n}_{G},\{v\Pi e^{n}_{G},-v\Pi e^{n}_{G}\},\Pi_{v}^{-}e^{n}_{P_{v}}\}$ and denoting $\boldsymbol{\Pi}\mathbf{e}=\{\Pi e^{k}_{G},\{\Pi_{x}e^{k}_{P_{x}},\Pi_{v}^{-}e^{k}_{P_{v}}\}\}_{k=1}^{\infty}$, one can obtain
	\begin{equation}\label{eqBerror}
		\begin{aligned}
			&\mathcal{B}_{n}(\boldsymbol{\Pi}\mathbf{e},\boldsymbol{\psi}_{h})\\
			=&\mathcal{B}_{n}(\mathbf{e},\boldsymbol{\psi}_{h})+\mathcal{B}_{n}(\boldsymbol{\rho},\boldsymbol{\psi}_{h})\\
			=&\mathcal{B}_{n}(\boldsymbol{\rho},\boldsymbol{\psi}_{h}),
		\end{aligned}
	\end{equation}
	where $\boldsymbol{\rho}=\{\rho^{k}_{G},\boldsymbol{\rho}^{k}_{\mathbf{P}}\}_{k=1}^{\infty}=\{\rho^{k}_{G},\{\rho^{k}_{P_{x}},\rho^{k}_{P_{v}}\}\}_{k=1}^{\infty}=\{\Pi G^{k}-G^{k},\{\Pi_{x} P^{k}_{x}-P^{k}_{x},\Pi_{v}^{-}P^{k}_{v}-P^{k}_{v}\}\}_{k=1}^{\infty}$.
	Similar to the derivations of Theorem \ref{thmstab}, we have
	\begin{equation*}
		\begin{aligned}
			&2\mathcal{B}_{n}(\boldsymbol{\Pi}\mathbf{e},\boldsymbol{\psi}_{h})\\
			\geq	&\sum_{k=0}^{n-1}d^{(\alpha)}_{k}\left(\sum_{i,j=1}^{N}\|\Pi e^{n-k}_{G}\|_{L^{2}(I_{i,j})}^{2}-\sum_{i,j=1}^{N}\|\Pi e^{0}_{G}\|_{L^{2}(I_{i,j})}^{2}\right)-\sum_{i,j=1}^{N}\|\Pi e^{n}_{G}\|_{L^{2}(I_{i,j})}^{2}\\
			&+\sum_{j=1}^{N}\int_{v_{j-1}}^{v_{j}}\left(\sum_{i=1}^{N-1}v\left (\Pi e^{n}_{G}(x_{i}^{-})-\Pi e^{n}_{G}(x_{i}^{+})\right )^{2}+v\left (\Pi e^{n}_{G}(x_{0}^{+})\right )^{2}+v\left (\Pi e^{n}_{G}(x_{N}^{-})\right )^{2}\right)dv\\
			&+\sum_{i=1}^{N}\int_{x_{i-1}}^{x_{i}}\left(\sum_{j=1}^{N-1}v_{j}\left (\Pi e^{n}_{G}(v_{j}^{-})-\Pi e^{n}_{G}(v_{j}^{+})\right )^{2}+v_{N}\left (\Pi e^{n}_{G}(v_{N}^{-})\right )^{2}+\left(\frac{2\vartheta}{h}\right)\left (\Pi e^{n}_{G}(v_{0}^{+})\right )^{2}\right)dx\\
			&+2\sum_{i,j=1}^{N}\|\Pi_{v}^{-}e^{n}_{P_{v}}\|_{L^{2}(I_{i,j})}^{2}.
		\end{aligned}
	\end{equation*}
	Furthermore, there holds
	\begin{equation*}
		\begin{aligned}
			\tau^{\alpha}\mathcal{B}_{n}(\boldsymbol{\rho},\boldsymbol{\psi}_{h})=&\tau^{\alpha}\sum_{i,j=1}^{N}\Bigg(\sum_{k=0}^{n-1}d^{(\alpha)}_{k}\left((\rho^{n-k}_{G}-\rho^{0}_{G}),\mu_{h}\right)_{I_{i,j}}+\left(v\rho^{n}_{P_{x}},\mu_{h}\right)_{I_{i,j}}-\left(v\rho^{n}_{P_{v}},\mu_{h}\right)_{I_{i,j}}\\
			&\qquad\qquad-\left (\frac{\partial}{\partial v}\rho^{n}_{P_{v}},\mu_{h}\right )_{I_{i,j}}+\int_{x_{i-1}}^{x_{i}}(\rho^{n}_{P_{v}}-\hat{\rho}^{n}_{P_{v}})\mu_{h}|_{v=v_{j-1}}^{v=v_{j}}dx-\left(\rho^{n}_{G},\mu_{h}\right)_{I_{i,j}}\\
			&\qquad\qquad-\left(\boldsymbol{\rho}^{n}_{\mathbf{P}},\boldsymbol{\nu}_{h}\right)_{I_{i,j}}+(\nabla \rho^{n}_{G},\boldsymbol{\nu}_{h})_{I_{i,j}}-(\rho^{n}_{G}-\hat{\rho}^{n}_{G},\mathbf{n}\cdot\boldsymbol{\nu}_{h})_{\partial I_{i,j}}\\
			&\qquad\qquad-(\rho^{n}_{P_{v}},\eta_{h})_{I_{i,j}}+\left (\frac{\partial}{\partial v} \rho^{n}_{G},\eta_{h}\right )_{I_{i,j}}-\int_{x_{i-1}}^{x_{i}}(\rho^{n}_{G}-\hat{\rho}^{n}_{G})\eta_{h}|_{v=v_{j-1}}^{v=v_{j}}dx\Bigg)\\
			\leq&\uppercase\expandafter{\romannumeral1}+\uppercase\expandafter{\romannumeral2}+\uppercase\expandafter{\romannumeral3}+\uppercase\expandafter{\romannumeral4}+\uppercase\expandafter{\romannumeral5}+\uppercase\expandafter{\romannumeral6}+\uppercase\expandafter{\romannumeral7},
		\end{aligned}
	\end{equation*}
	where
	\begin{equation*}
		\begin{aligned}
			\uppercase\expandafter{\romannumeral1}=&\tau^{\alpha}\sum_{i,j=1}^{N}\Bigg(\left(v\rho^{n}_{P_{x}},\mu_{h}\right)_{I_{i,j}}-\left(v\rho^{n}_{P_{v}},\mu_{h}\right)_{I_{i,j}}-\left(\boldsymbol{\rho}^{n}_{\mathbf{P}},\boldsymbol{\nu}_{h}\right)_{I_{i,j}}-\left(\rho^{n}_{G},\mu_{h}\right)_{I_{i,j}}\Bigg),\\
			\uppercase\expandafter{\romannumeral2}=&\tau^{\alpha}\sum_{i,j=1}^{N}\Bigg(-\int_{x_{i-1}}^{x_{i}}\hat{\rho}^{n}_{P_{v}}\mu_{h}|_{v=v_{j-1}}^{v=v_{j}}dx\Bigg),\\
			\uppercase\expandafter{\romannumeral3}=&-\tau^{\alpha}\sum_{i,j=1}^{N}\Bigg(( \rho^{n}_{G},\nabla\cdot\boldsymbol{\nu}_{h})_{I_{i,j}}-(\hat{\rho}^{n}_{G},\mathbf{n}\cdot\boldsymbol{\nu}_{h})_{\partial I_{i,j}}\Bigg),\\
			\uppercase\expandafter{\romannumeral4}=&-\tau^{\alpha}\sum_{i,j=1}^{N}\Bigg(\left ( \rho^{n}_{G},\frac{\partial}{\partial v}\eta_{h}\right )_{I_{i,j}}-\int_{x_{i-1}}^{x_{i}}\hat{\rho}^{n}_{G}\eta_{h}|_{v=v_{j-1}}^{v=v_{j}}dx\Bigg),\\
			\uppercase\expandafter{\romannumeral5}=&\tau^{\alpha}\sum_{i,j=1}^{N}\Bigg(\left (\rho^{n}_{P_{v}},\frac{\partial}{\partial v}\mu_{h}\right )_{I_{i,j}}\Bigg),\\
			\uppercase\expandafter{\romannumeral6}=&\tau^{\alpha}\sum_{i,j=1}^{N}\Bigg(-(\rho^{n}_{P_{v}},\eta_{h})_{I_{i,j}}\Bigg),\\
			\uppercase\expandafter{\romannumeral7}=&\tau^{\alpha}\sum_{i,j=1}^{N}\sum_{k=0}^{n-1}d^{(\alpha)}_{k}\left((\rho^{n-k}_{G}-\rho^{0}_{G}),\mu_{h}\right)_{I_{i,j}}.
		\end{aligned}
	\end{equation*}
	The property of the projection \citep[see, e.g.,][]{Dong.2009AoaLDGMfLTDFOP}, the Cauchy-Schwarz inequality, and the Young inequality show that for $\epsilon>0$,
	\begin{equation*}
		\begin{aligned}
			\uppercase\expandafter{\romannumeral1}
			\leq& C\epsilon^{-1}\tau^{\alpha}h^{2k+2}+\epsilon\tau^{\alpha}\|\Pi e^{n}_{G}\|_{L^{2}(\Omega)}^{2}.
		\end{aligned}
	\end{equation*}
	According to \eqref{eqfluxv2},  for $\epsilon>0$, there exists 
	\begin{equation*}
		\uppercase\expandafter{\romannumeral2}\leq C\tau^{\alpha}\epsilon^{-1} h^{2k+2}+\tau^{\alpha}\epsilon\sum_{i=1}^{N}\int_{x_{i-1}}^{x_{i}}\left(\frac{2\vartheta}{h}\right )(\Pi e^{n}_{G}(v_{0}^{+}))^{2}dx.
	\end{equation*}
	According to Lemma \ref{lemspec}, for $\epsilon>0$, one has
	\begin{equation*}
		\begin{aligned}
			\uppercase\expandafter{\romannumeral3}\leq & C\tau^{\alpha}\sum_{i,j=1}^{N}\Bigg|( \rho^{n}_{G},\nabla\cdot\boldsymbol{\nu}_{h})_{I_{i,j}}-(\hat{\rho}^{n}_{G},\mathbf{n}\cdot\boldsymbol{\nu}_{h})_{\partial I_{i,j}}\Bigg|                                                                                  \\
			\leq                                       & C\tau^{\alpha}\sum_{i,j=1}^{N}\Bigg|\left( v\rho^{n}_{G},\left (\frac{\partial}{\partial x} \Pi e^{n}_{G}- \frac{\partial}{\partial v} \Pi e^{n}_{G}\right )\right )_{I_{i,j}}-(\hat{\rho}^{n}_{G},\mathbf{n}\cdot\boldsymbol{\nu}_{h})_{\partial I_{i,j}}\Bigg| \\
			& +C\tau^{\alpha}\sum_{i,j=1}^{N}\Bigg|\left( \rho^{n}_{G}, \Pi e^{n}_{G}\right )_{I_{i,j}}\Bigg|                                                                                                                                                                  \\
			\leq                                       & C\epsilon^{-1}\tau^{\alpha}h^{2k+2}+\tau^{\alpha}\epsilon\| \Pi e^{n}_{G}\|_{L^{2}(\Omega)}^{2}.
		\end{aligned}
	\end{equation*}
	Similarly, for $\epsilon>0$,  there holds
	\begin{equation*}
		\uppercase\expandafter{\romannumeral4}\leq C\epsilon^{-1}\tau^{\alpha}h^{2k+2}+\tau^{\alpha}\epsilon\|\Pi_{v}^{-}e^{n}_{P_{v}}\|_{L^{2}(\Omega)}^{2}.
	\end{equation*}
	The property of projection \citep[see, e.g.,][]{Dong.2009AoaLDGMfLTDFOP} implies
	\begin{equation*}
		\uppercase\expandafter{\romannumeral5}=0.
	\end{equation*}
	As for $\uppercase\expandafter{\romannumeral6}$, the Cauchy-Schwarz inequality, the Young inequality, and the property of projection lead to that for $\epsilon>0$,
	\begin{equation*}
		\begin{aligned}
			&\uppercase\expandafter{\romannumeral6}\leq C\epsilon^{-1}\tau^{\alpha}h^{2k+2}+\tau^{\alpha}\epsilon\|\Pi^{-}_{v}e^{n}_{P_{v}}\|^{2}_{L^{2}(\Omega)}.
		\end{aligned}
	\end{equation*}
	Using the definition of $d_{k}^{(\alpha)}$ and Lemma \ref{lempropb}, we have, for $\epsilon>0$,
	\begin{equation*}
		\uppercase\expandafter{\romannumeral7}\leq C\epsilon^{-1}h^{2k+2}+\epsilon\|\Pi e^{n}_{G}\|^{2}_{L^2(\Omega)}.
	\end{equation*}
	Combining \eqref{eqBerror} and choosing $\epsilon>0$ small enough, we can obtain
	\begin{equation*}
		\begin{aligned}
			\tau^{\alpha}\sum_{k=0}^{n-1}d^{(\alpha)}_{k}\left(\sum_{i,j=1}^{N}\|\Pi e^{n-k}_{G}\|_{L^{2}(I_{i,j})}^{2}-\sum_{i,j=1}^{N}\|\Pi e^{0}_{G}\|_{L^{2}(I_{i,j})}^{2}\right)\qquad\qquad\qquad&\\-\left (\frac{3}{2}\tau^{\alpha}+\epsilon\right )\sum_{i,j=1}^{N}\|\Pi e^{n}_{G}\|_{L^{2}(I_{i,j})}^{2}&\leq Ch^{2k+2},
		\end{aligned}
	\end{equation*}
	which leads to
	\begin{equation*}
		\begin{aligned}
			\tau^{\alpha}\sum_{k=0}^{n-1}d^{(\alpha)}_{k}\left(\sum_{i,j=1}^{N}\|\Pi e^{n-k}_{G}\|_{L^{2}(\Omega)}^{2}-\sum_{i,j=1}^{N}\|\Pi e^{0}_{G}\|_{L^{2}(I_{i,j})}^{2}\right)\qquad&\\
			-\left (\frac{3}{2}\tau^{\alpha}+\epsilon\right )\sum_{i,j=1}^{N}\left (\|\Pi e^{n}_{G}\|_{L^{2}(I_{i,j})}^{2}-\|\Pi e^{0}_{G}\|_{L^{2}(I_{i,j})}^{2}\right )&\leq Ch^{2k+2}.
		\end{aligned}
	\end{equation*}
	Further, the definition of $d^{(\alpha)}_{k}$ yields
	\begin{equation*}
		\|\Pi e^{n-k}_{G}\|_{L^{2}(\Omega)}^{2}\leq Ch^{2k+2}.
	\end{equation*}
	Thus combining the above estimate and the property of projection leads to the desired result.  
\end{proof}

\section{Numerical Experiments}
In this section, we provide two examples to verify the temporal and spatial convergence rates, respectively. In the following, we take $\vartheta=1$.
\begin{example}
	Here we show temporal convergence rates for the system (\ref{eqretosol}) with boundary condition (\ref{eqbdcond}). We take $T=1$, $k=1$, $h=1/16$, and for the initial condition consider the following three cases, i.e.,
	\begin{enumerate}[(a)]
		\item\label{exconda}
		\begin{equation*}
			G_{0}=x\times\sin(\pi v),\quad f=0;
		\end{equation*}
		\item\label{excondb}
		\begin{equation*}
			G_{0}=\chi_{(0.5,1)}(x)\chi_{(0,0.5)}(v),\quad f=0;
		\end{equation*}
		\item\label{excondc}
		\begin{equation*}
			G_{0}=0,\quad f=\chi_{(0.5,1)}(x)\chi_{(0,0.5)}(v)t^{0.8},
		\end{equation*}
	\end{enumerate}
	where $\chi_{(a,b)}$ is the characteristic function on $(a,b)$.
	Because the exact solution is unknown, the errors and convergence rates can be calculated by
	\begin{equation*}
		E_{\tau}=\|G_{\tau}-G_{\tau/2}\|_{L^{2}(\Omega)},\quad {\rm rate}=\frac{\ln(E_{\tau}/E_{\tau/2})}{\ln(2)},
	\end{equation*}
	where $G_{\tau}$ denotes the numerical solution of $G$ at time $T$ with step size $\tau$.
	
	We first provide the numerical results for the system (\ref{eqretosol}) with boundary condition (\ref{eqbdcond}) and initial condition \eqref{exconda}. From the errors and convergence rates shown in Table \ref{tab:timesmooth}, it can be noted that all the results agree with Theorem \ref{thmsemierr}. Then we consider the system with the initial condition \eqref{excondb}. Although the initial condition is nonsmooth, the convergence rates shown in Table \ref{tab:timenonsmooth} are still $\mathcal{O}(\tau)$, which validate the results of Theorem \ref{thmsemierr}. Next, for the initial condition \eqref{excondc}, the corresponding convergence rates presented in Table \ref{tab:timenonhom} are same with the predicted ones in Theorem \ref{thmsemierr}.
	\begin{table}[htbp]
		\caption{Temporal errors and convergence rates for the system (\ref{eqretosol}) with boundary condition (\ref{eqbdcond}) and initial condition \eqref{exconda}}
		\begin{center}
			
			\begin{tabular}{cccccc}
				\hline
				$\alpha\backslash 1/\tau$	& 10 & 20 & 40 & 80 & 160   \\
				\hline
				0.3 & 2.726E-04 & 1.329E-04 & 6.564E-05 & 3.262E-05 & 1.626E-05   \\
				& Rate & 1.0360  & 1.0179  & 1.0089  & 1.0045    \\
				0.5 & 4.442E-04 & 2.138E-04 & 1.049E-04 & 5.197E-05 & 2.586E-05   \\
				& Rate & 1.0550  & 1.0272  & 1.0135  & 1.0067   \\
				0.8 & 5.479E-04 & 2.487E-04 & 1.187E-04 & 5.803E-05 & 2.869E-05   \\
				& Rate & 1.1395  & 1.0668  & 1.0326  & 1.0161   \\
				\hline
			\end{tabular}
		\end{center}
		\label{tab:timesmooth}
	\end{table}
	\begin{table}[htbp]
		\caption{Temporal errors and convergence rates for the system (\ref{eqretosol}) with boundary condition (\ref{eqbdcond}) and initial condition \eqref{excondb}}
		\begin{center}
			
			\begin{tabular}{cccccc}
				\hline
				$\alpha\backslash 1/\tau$& 10 & 20 & 40 & 80 & 160 \\
				\hline
				0.2 & 1.264E-04 & 6.192E-05 & 3.065E-05 & 1.525E-05 & 7.603E-06 \\
				& Rate & 1.0294  & 1.0147  & 1.0073  & 1.0037  \\
				0.4 & 2.531E-04 & 1.227E-04 & 6.041E-05 & 2.997E-05 & 1.493E-05 \\
				& Rate & 1.0447  & 1.0222  & 1.0110  & 1.0055  \\
				0.6 & 3.521E-04 & 1.677E-04 & 8.184E-05 & 4.044E-05 & 2.010E-05 \\
				& Rate & 1.0705  & 1.0346  & 1.0172  & 1.0085  \\
				\hline
			\end{tabular}
		\end{center}
		\label{tab:timenonsmooth}
	\end{table}

	\begin{table}[htbp]
		\caption{Temporal errors and convergence rates for the system (\ref{eqretosol}) with boundary condition (\ref{eqbdcond}) and initial condition \eqref{excondc}}
		\begin{center}
			\begin{tabular}{ccccccc}
				\hline
				$\alpha\backslash 1/\tau$& 10 & 20 & 40 & 80 & 160   \\
				\hline
				0.2 & 6.500E-06 & 3.352E-06 & 1.705E-06 & 8.610E-07 & 4.329E-07   \\
				& Rate & 0.9556  & 0.9751  & 0.9858  & 0.9919    \\
				0.5 & 9.831E-06 & 5.076E-06 & 2.584E-06 & 1.306E-06 & 6.567E-07   \\
				& Rate & 0.9537  & 0.9740  & 0.9850  & 0.9913    \\
				0.7 & 6.149E-06 & 3.191E-06 & 1.630E-06 & 8.253E-07 & 4.158E-07   \\
				& Rate & 0.9463  & 0.9693  & 0.9818  & 0.9890    \\
				\hline
			\end{tabular}
		\end{center}
		\label{tab:timenonhom}
	\end{table}
	
\end{example}

\begin{example}
	Here we validate the spatial convergence rates for the scheme (\ref{eqfullscheme}). We take $T=1$,
	\begin{equation*}
		G_{0}=\sin(\pi x)\sin(\pi v)
	\end{equation*}
	and
	\begin{equation*}
		\begin{aligned}
			f=&\Gamma(\alpha+1)\sin(\pi x)\sin(\pi v)\\
			&+\Bigg(\pi^{2}\sin(\pi x)\sin(\pi v)+v\pi\cos(\pi x)\sin(\pi v)\\
			&-v\pi\sin(\pi x)\cos(\pi v)-\sin(\pi x)\sin(\pi v)\Bigg)t^{\alpha+1},
		\end{aligned}
	\end{equation*}
which implies the exact solution
	\begin{equation*}
		G=(t^{\alpha}+1)\sin(\pi x)\sin(\pi v).
	\end{equation*}
	
	In Table \ref{tab:spaO1}, we choose the order of approximation polynomial $k=1$ and $\tau=\frac{1}{100}$. We find the convergence rates are $\mathcal{O}(h^{2})$, which are same with the theoretical ones in Theorem \ref{thmfullerr}. In Table \ref{tab:spaO2}, we take the order of approximation polynomial $k=2$ and to investigate the convergence in space and eliminate the influence from temporal discretization, we take $\tau=\frac{1}{200}$. From Table  \ref{tab:spaO2}, we see the convergence rates are $\mathcal{O}(h^{3})$, which are consistent with the predicted ones.
	\begin{table}[htbp]
		\caption{Spatial errors and convergence rates for the scheme (\ref{eqfullscheme}) with the order of approximation polynomial $k=1$}
		\begin{center}
			\begin{tabular}{cccccc}
				\hline
				$\alpha\backslash 1/h$	& 4 & 8 & 12 & 16 & 20 \\
				\hline
				0.3 & 1.032E-01 & 2.625E-02 & 1.175E-02 & 6.635E-03 & 4.259E-03 \\
				& Rate & 1.9755  & 1.9830  & 1.9859  & 1.9867  \\
				0.5 & 1.032E-01 & 2.623E-02 & 1.174E-02 & 6.627E-03 & 4.252E-03 \\
				& Rate & 1.9757  & 1.9835  & 1.9869  & 1.9884  \\
				0.7 & 1.031E-01 & 2.622E-02 & 1.173E-02 & 6.621E-03 & 4.247E-03 \\
				& Rate & 1.9758  & 1.9839  & 1.9876  & 1.9896  \\
				\hline
			\end{tabular}
		\end{center}
		\label{tab:spaO1}
	\end{table}
	
	\begin{table}[htbp]
		\caption{Spatial errors and convergence rates for the scheme (\ref{eqfullscheme}) with the order of approximation polynomial $k=2$}
		\begin{center}
			\begin{tabular}{cccccc}
				\hline
				$\alpha\backslash 1/h$	& 4 & 8 & 12 & 16 & 20 \\
				\hline
				0.4 & 3.372E-03 & 4.285E-04 & 1.269E-04 & 5.365E-05 & 2.779E-05 \\
				& Rate & 2.9763  & 3.0014  & 2.9927  & 2.9477  \\
				0.6 & 3.371E-03 & 4.281E-04 & 1.266E-04 & 5.331E-05 & 2.730E-05 \\
				& Rate & 2.9770  & 3.0048  & 3.0069  & 2.9989  \\
				0.8 & 3.370E-03 & 4.280E-04 & 1.265E-04 & 5.321E-05 & 2.719E-05 \\
				& Rate & 2.9772  & 3.0058  & 3.0101  & 3.0089  \\
				\hline
			\end{tabular}
		\end{center}
		\label{tab:spaO2}
	\end{table}
\end{example}

\section{Conclusions}
We provide the regularity estimates and numerical analyses for fractional Klein-Kramers equation. By introducing a new positive operator $\mathcal{L}$ to overcome hypocoercivity of the original one and building its resolvent estimate, we give spatial and temporal regularity of the solution. Then, backward Euler convolution quadrature method and local discontinuous Galerkin method are used to approximate Riemann-Liouville fractional derivative and the operator $\mathcal{L}$, respectively, and the complete error analyses are also built.  Finally,  we perform the numerical experiments, which support the theoretical results.

\section*{Acknowledgements}

This work was supported by National Natural Science Foundation of China under Grant No. 12071195, AI and Big Data Funds under Grant No. 2019620005000775, and Fundamental Research Funds for the Central Universities under Grant Nos. lzujbky-2021-it26 and lzujbky-2021-kb15.

\bibliographystyle{IMANUM-BIB}
\bibliography{references}
\clearpage
\end{document}